\documentclass[onecolumn]{autart}
\usepackage{}    % Enable this line and disable the
\usepackage{amsmath,bm,times}
\usepackage{mathrsfs}
\usepackage{bbm}    % Enable this line and disable the
\usepackage{graphicx}           % Include figure files
\usepackage{dcolumn}            % Align table columns on decimal point
\usepackage{bm}                 % bold math

\usepackage{graphics}           % for pdf, bitmapped graphics files
\usepackage{epsfig}             % for postscript graphics files
\usepackage{times}              % assumes new font selection scheme installed
\usepackage{amsmath}
\usepackage{amssymb}
\usepackage{amsfonts}
\usepackage{fancyhdr}
\usepackage{color}
\usepackage{subfigure}
\usepackage{enumerate}
\usepackage{cite}
\usepackage{wrapfig}
\usepackage{url}

\usepackage{bbm}
\usepackage[mathscr]{euscript}
\usepackage{graphicx}
\usepackage{amsmath,amssymb}
\usepackage{bm}

\usepackage{pstricks}
\usepackage{pst-node}

\usepackage{tikz}
\usetikzlibrary{arrows,chains,matrix,positioning,scopes}
\tikzstyle{element}=[rectangle,draw,fill=white, line width=1pt]
\tikzstyle{terminal}=[circle,draw, scale=0.3, line width=1pt,red]
\tikzstyle{fleche}=[->,>=stealth', very thick]
\tikzstyle{fleche1}=[->,>=stealth', very thick, red]
\usepackage{placeins}

\def\R{{\mathbb{R}}}

\def\la{\lambda}

\def\calL{{\mathcal{L}}}

\def\calT{{\mathcal{T}}}

\def\calB{{\mathcal{B}}}

\def\calA{{\mathcal{A}}}

\def\calF{{\mathcal{F}}}

\def\R{\mathbb R}

\def\C{\mathbb C}

\def\N{\mathbb N}

\newtheorem{example}{Example}[section]
\newtheorem{mainassumptions}{Assumption}[section]

\definecolor{darkgoldenrod4}{rgb}{0.55,0.4,0.55}
\definecolor{maroon4}{rgb}{0.55,0.11,0.38}
\definecolor{indianred}{rgb}{0.8,0.36,0.36}
\definecolor{purple1}{rgb}{0.61,0.19,1}
\definecolor{goldenrod1}{rgb}{1,0.76,0.15}
\definecolor{indianred3}{rgb}{0.8,0.33,0.33}
\definecolor{red4}{rgb}{0.55,0,0}
\definecolor{darkslategray}{rgb}{0.18,0.31,0.31}
\definecolor{firebrick}{rgb}{0.7,0.13,0.13}
\definecolor{slateblue3}{rgb}{0.41,0.35,0.8}
\definecolor{mediumorchid4}{rgb}{0.48,0.22,0.55}
\definecolor{thistle4}{rgb}{0.55,0.48,0.55}
\definecolor{rltred}{rgb}{0.75,0,0}
\definecolor{rltgreen}{rgb}{0,0.5,0}
\definecolor{oneblue}{rgb}{0,0,0.75}
\definecolor{marron}{rgb}{0.64,0.16,0.16}
\definecolor{forestgreen}{rgb}{0.13,0.54,0.13}
\definecolor{purple}{rgb}{0.62,0.12,0.94}
\definecolor{dockerblue}{rgb}{0.11,0.56,0.98}
\definecolor{freeblue}{rgb}{0.25,0.41,0.88}
\definecolor{myblue}{rgb}{0,0.2,0.4}

\newtheorem{theorem}{Theorem}[section]  %
\newtheorem{lemma}{Lemma}[section]
\newtheorem{corollary}{Corollary}[section]

\newtheorem{definition}{Definition}[section]

\newtheorem{remark}{Remark}[section]

\numberwithin{equation}{section}

\newenvironment{proof}{{\it Proof: \enspace}}{\hfill $\blacksquare$\par}

\allowdisplaybreaks[4] % [1] page breaks are allowed, but avoided if possible

\begin{document}
%\linenumbers
%\begin{frontmatter}
%\runtitle{Insert a suggested running title}  % Running title for regular
                                              % papers but only if the title
                                              % is over 5 words. Running title
                                              % is not shown in output.
%
\begin{frontmatter}
%\runtitle{Insert a suggested running title}  % Running title for regular
                                              % papers but only if the title
                                              % is over 5 words. Running title
                                              % is not shown in output.

\title{A Spectral-based ISS small-gain theorem  for   boundary control systems with infinite couplings} % Title, preferably not more
                                              % than 10 words.

%\thanks[footnoteinfo]{Corresponding author:Jun Zheng}
%\thanks{This work was supported in part by the National Natural Science Foundation of China under Grant 11901482}

\author{Yassine El Gantouh$^1$}\ead{elgantouhyassine@gmail.com},
\author{Jun Zheng$^1$}\ead{zhengjun2014@aliyun.com},
\author{Guchuan Zhu$^2$}\ead{guchuan.zhu@polymtl.ca},\ and
\author{Dingshi Li$^1$}\ead{lidingshi2006@163.com}

\address{$^1${School of Mathematics}, Southwest Jiaotong University, Chengdu, 611756,  China}  % Please supply
\address{$^2$Department of Electrical Engineering, Polytechnique Montr\'{e}al, Montreal, H3T 1J4, Canada}

\begin{abstract}                           % Abstract of not more than 200 words.
 We study the input-to-state stability (ISS) of boundary control systems allowing for infinitely many boundary couplings. Using semigroup perturbation theory and the theory of positive linear operators on Banach lattices, we derive a spectral small-gain condition ensuring exponential ISS. We further investigate linear Boltzmann-type equations on an infinite network of intersecting circles, incorporating delays, scattering, and disturbances acting at the junction. For this class of systems, we prove that a spectral small-gain condition on the transmission operator matrix guarantees exponential ISS with respect to disturbances propagating through the network. Moreover, we derive explicit ISS estimates for {certain} classes of dynamical processes. Finally, we demonstrate the practical applicability of our results by considering two important classes of time-delayed transmission conditions.
\end{abstract}

\vspace*{-12pt}
\end{frontmatter}
\endNoHyper
%\linenumbers

%\tableofcontents %
%%%%%%%%%%%%%%%%%%%%%%%%%%%%%%%%%%%%%%%%%%%%%%%%%%%%%%%%%%%%%%%%%%%%%%%%%%%%%%%%%%%%%%%%%%%%%%%%%

 \section{Introduction}
    Over the past few decades, models based on partial differential equations (PDEs) have proven highly effective for analyzing flows on networks in {fields such as} irrigation, gas pipelines, blood circulation, vehicular traffic, supply chains (see, e.g., \cite{BCGHP} for a survey). In these models, each PDE is defined on a one-dimensional domain corresponding to an edge of a metric graph, and the dynamics on different edges are coupled through transmission conditions (TCs) imposed at the vertices. The design and analysis of TCs not only critically influence the overall system dynamics but also represent a major source of mathematical complexity (see, e.g., \cite{BRT,BCGHP,YGS,DZ,Dorn,HLS} and references therein). Notably, TCs are inherently sensitive to external disturbances and internal uncertainties, making the development of robust control and observation strategies essential for reliable network operation.

    A rigorous framework for analyzing stability with respect to disturbances for interconnected systems of ordinary differential equations (ODEs) is provided by input-to-state stability (ISS) theory, introduced by Sontag in the late 1980s \cite{Sontag}. ISS unifies Lyapunov and input-output stability, providing a cohesive treatment of both internal dynamics and external disturbances. This framework has become a cornerstone of robust stability analysis, with applications in robust stabilization, nonlinear observer design, and the stability of large-scale networked systems (see, e.g., \cite{KKK,Sontag1} and references therein). A particularly powerful tool within ISS theory is the small-gain approach, which establishes the stability of interconnected systems by verifying the ISS of individual subsystems and imposing a small-gain condition (SGC) on their interconnection gains \cite{DRW1,JTP}.

   While the small-gain framework is well understood for finite interconnections of ODEs, many applications involve infinite network structures, {for which} the analysis becomes significantly more challenging. The main difficulty arises from the fact that {in this case,} the gain operator acts on an infinite-dimensional space. As shown in \cite{DMSW}, the classic max-form SGCs developed for finite-dimensional systems \cite{DRW1} fail to ensure stability for infinite networks, even in the linear setting. To overcome this limitation, more restrictive robust strong SGCs were proposed in \cite{DMSW,DP}. For infinite networks composed of exponentially input-to-state stable (ISS) subsystems with linear gains, it was proved  in \cite{KMSNZ} that if the spectral radius of the gain operator is less than one, the network is exponentially ISS. Moreover, the system admits a coercive exponential ISS Lyapunov function. This Lyapunov-based framework was later extended to nonlinear gains in \cite{KMZ}.
	
   Building on these developments, ISS small-gain analysis has also been extended to finite couplings of infinite-dimensional systems, including PDEs \cite{DM,KK4,KK5,MI,Andri1}. A fundamental distinction in PDE stability analysis concerns the spatial location of couplings and disturbances, which has no direct analogue for ODEs. For PDEs involving only in-domain couplings or disturbances, the ISS can often be directly established by extending Lyapunov methods from finite-dimensional systems \cite{DM}. In contrast, boundary couplings and disturbances introduce substantial analytical challenges, as their action is modeled by unbounded operators, which complicates the direct application of classical Lyapunov techniques (see, e.g., \cite{JaNPS,KK1,KK2,LSZS,MKK,ZG3,ZG4}). Surveys and summaries on ISS analysis for PDEs and infinite-dimensional systems can be found in \cite{KK3,Andri,ZG2,ZG6}.

   %In particular, their action is modeled by unbounded operators in the abstract formulation, complicating the direct application of Lyapunov techniques; see  \cite{ZG2,ZG6} for a summary of  the methods used for ISS analysis of different systems  and \cite{KK3,Andri} for surveys on the development of ISS theory for PDEs and infinite-dimensional systems, respectively.

    Recently, ISS small-gain analysis for infinite-dimensional systems with infinitely many couplings was considered in \cite{MKG}, {where} a nonlinear small-gain theorem was established, showing that the network is ISS if the corresponding nonlinear gain operator satisfies the monotone limit property. However, the coupling structures considered in \cite{MKG} are restricted to semimaximum or summation forms and therefore do not cover PDE systems with nonlocal or boundary couplings. Moreover, for many infinitely coupled PDE systems, well-posedness and stability cannot be assumed {\textit{a priori}} and typically require careful case-by-case analysis.

    Motivated by the ISS analysis of  infinitely boundary-coupled PDEs, this paper investigates the ISS of a general class of boundary control systems.  The considered framework encompasses a broad range of linear time-invariant (LTI) infinite-dimensional systems. A key feature of {this} framework is that it allows the modeling of PDEs with infinitely many boundary or in-domain couplings. It also covers several important classes of systems, including PDEs with dynamic boundary conditions, models of population dynamics with unbounded birth processes, and time-delay systems.
    % The considered framework encompasses a broad range of linear time-invariant (LTI) infinite-dimensional systems, including PDEs with dynamic boundary conditions, models of population dynamics with unbounded birth processes, time-delay systems, and PDEs with infinitely many boundary or in-domain couplings.
    A comprehensive characterization of ISS for LTI infinite-dimensional systems with unbounded input operators in a general setting was established in \cite{JaNPS}. In particular, it was shown that the ISS for such systems is equivalent to the exponential stability of the underlying semigroup together with the admissibility of the input operator in the sense of \cite{WC}. Although this characterization is mathematically general, verifying the admissibility condition in practice is often difficult, {because} it depends on detailed structural properties of the semigroup that are rarely available in an explicit form. {To address this issue,} one of the main objectives of this paper is to provide, for a general class of boundary control systems, a spectral SGC that guarantees the ISS without requiring the verification of additional admissibility conditions.

    A major difficulty in establishing the ISS property for this class of boundary control systems stems from the unbounded and non-closable nature of the operator describing an infinite number of boundary couplings. More precisely, this unboundedness implies that, when the boundary control system is reformulated as a control system with distributed input, the resulting system is subject to a Weiss-Staffans type perturbation \cite{WF,Staffans}. As a consequence, the state evolves in two distinct extrapolation spaces associated with the original state space, which prevents the direct application of existing ISS characterizations such as those in \cite{JaNPS}. One possible approach to overcoming this difficulty is to construct a suitable Banach space that is contained in the intersection of these two extrapolation spaces and contains the range of the input operator. An elegant construction of such spaces was developed in \cite{WF} within the framework of regular linear systems. However, this framework requires verifying that an abstract quadruple of operator families satisfies certain functional equations, a task that is itself highly nontrivial; see, for example, \cite{WF} for the Hilbert space case and \cite[Chap.~7]{Staffans} for the general Banach space setting.

     To overcome these analytical difficulties, we propose a purely operator-theoretic approach for studying the ISS of this general class of boundary control systems. First, we employ a perturbation technique inspired by \cite{Gr,Salam} to transform the original boundary control system into a system governed by an operator with a perturbed domain (in the sense of \cite{Gr}) and an unbounded input operator. Subsequently, in the spirit of \cite{WF}, we construct a Banach space contained in the intersection of the two relevant extrapolation spaces that also contains the range of the input operator, independently of the abstract framework developed in \cite{Staffans,WF,WT}. This construction renders the admissibility property well-defined within the perturbed framework and enables both well-posedness and ISS analysis. Our approach exploits the inherent positivity of the unforced  dynamics and combines perturbation theory for positive semigroups \cite{El1}, resolvent-based admissibility criteria \cite{ELLC}, and the theory of positive linear operators on Banach lattices. In particular, we establish a SGC through a spectral characterization of ISS for this class of boundary control systems and derive an explicit ISS estimate. This alternative perspective provides new insights into the well-posedness and ISS analysis of boundary control systems.

   {Based on} the development of the theoretical framework, we investigate the ISS of linear Boltzmann-like equations on an infinite network of intersecting circles governed by boundary-coupled first-order PDEs with boundary input. Building on models with simpler TCs \cite{YGS}, we consider a more general setting incorporating distributed delays, scattering effects, and vertex inputs, which lead to nonlocal TCs in both time and space at the network junction. Using our abstract framework, we derive a spectral SGC ensuring the ISS without imposing assumptions on the ratios of edge lengths or delays. This extends existing stability criteria such as those in \cite{YGS} and provides a complete and verifiable condition for robust stability {of} networked transport systems. Furthermore, we establish an explicit ISS estimate that quantitatively characterizes the stability properties of the system in the presence of delays and scattering effects. Consequently, our results provide an alternative approach to proving the ISS for infinite networks of linear PDEs with nonlocal and boundary couplings, thereby bypassing the restrictive assumptions imposed in \cite{MKG}.

    Overall, the main contributions of this work are twofold:
    \begin{itemize}
    \item[(i)] We develop a {spectral-based small-gain theorem} for boundary control systems by integrating perturbation theory with positivity-based methods. This yields an alternative ISS criterion for infinite-dimensional systems that accommodates general coupling formulations.

   \item[(ii)] We apply the proposed framework to infinite networks governed by boundary-coupled PDEs with delays and scattering. In this setting, we derive a spectral small-gain condition without geometric or delay constraints together with explicit ISS estimates, thereby generalizing existing stability results for networked systems.
   \end{itemize}

    In the rest of this paper, we first introduce the notation used throughout. In Section~\ref{Sec:2}, we recall basic results on positive semigroups and provide an overview of boundary control systems, along with essential theoretical concepts. Section~\ref{Sec:2.3} is devoted to the problem formulation and the statement of the main results. In Section~\ref{Sec:3}, we present detailed proofs of these results. In Section~\ref{Sec:4}, we apply our findings to investigate the ISS of linear Boltzmann-like equations on an infinite network of intersecting circles, incorporating delays, scattering effects, and vertex disturbances at the junctions. Finally, in Appendix~\ref{App1}, we prove tow technical lemmas used in the proof of the main results.

	{\bf Notation}. Throughout this paper, $\C$, $\R$, $\R_+$, and $\N$ denote the sets of complex numbers, real numbers, positive real numbers, and natural numbers, respectively.  For a Banach space $E$, an interval $I \subset \R$, and $p \in [1,\infty]${:}
	\begin{itemize}
		\item  ${L}^p(I;E)$ denotes the space of (equivalence classes of) Bochner measurable functions $f \colon I \to E$ such that $	\Vert f\Vert_{{ L}^p(I;E)} := \left(\int_I \Vert f(x)\Vert_E^{p}dx\right)^{1/p} <\infty$ for $p<\infty$,
		and $ \Vert f\Vert_{{ L}^\infty(I;E)}:= \operatorname{ess\,sup}_{x\in I} \Vert f(x)\Vert_E$ for $p=\infty$. When $I =(a,b)$ with $a,b\in\mathbb{R}$ satisfying  $a<b$, we simply write ${ L}^p(a,b;E)$.
		\item  ${C}(I; E)$ denotes the space of continuous functions from $I$ to $E$, equipped with the supremum norm.
		%\item ${W}^{1,p}(I;E)$ denotes the Sobolev space of absolutely continuous functions $f \colon I \to E$ whose distributional derivative $\partial_xf$ belongs to ${ L}^p(I;E)$, endowed with the norm
		%\begin{align*}
		%$\Vert f\Vert_{{W}^{1,p}(I;E)} := \Vert f\Vert_{{ L}^p(I;E)} + \Vert \partial_xf\Vert_{{ L}^p(I;E)} .$
		%\end{align*}
	\end{itemize}
	  %, and $W^{1,p}(\R_+; E)$.
	
	Let $E$ and $F$ be Banach spaces. {We} denote by $\mathcal{L}(E,F)$ the space of bounded linear operators from $E$ to $F$, and write $\mathcal{L}(E) = \mathcal{L}(E,E)$. For $P \in \mathcal{L}(E,F)$ and a subspace $Z \subset E$, {we denote by} $P\vert_{Z}$ the restriction of $P$ to $Z$; $\operatorname{ran} P$ and $\ker P$ denote its range and kernel, respectively. For $P \in \mathcal{L}(E)$, $r(P)$ denotes its spectral radius. The bracket $\langle \cdot,\cdot\rangle_{E,E'}$ denotes the duality pairing, where $E'$ is the topological dual of $E$. The symbol $\mathcal{I}$ denotes the identity operator (context will clarify).
	
	Let $E = (E, \le)$ be a \emph{Banach lattice}, i.e., a partially ordered Banach space in which every pair of elements $x, y \in E$ has a supremum $x \vee y$, and the following properties hold for all $x, y, z \in E$ and $\alpha \ge 0$:
	\begin{itemize}
		\item $x \le y$ implies $x + z \le y + z$ and $\alpha x \le \alpha y$,
		\item $|x| \le |y|$ implies $\|x\| \le \|y\|$,
	\end{itemize}
	where $|x| = x \vee (-x)$ denotes the absolute value of $x$. The set $E_+ = \{x \in E : 0\le x\}$ forms the \emph{positive cone}. This notation is also used for general ordered Banach spaces. An operator $P \in \mathcal{L}(E,F)$ between Banach lattices is called \emph{positive} if $P E_+ \subset F_+$, and we denote the set of such operators by $\mathcal{L}_+(E,F)$. The topological dual $E'$ of a Banach lattice $E$, endowed with the dual norm and the dual order, is also a Banach lattice. For further details; see, e.g., \cite{CHARALAMBOS,Schaf}.
	
	Let $\ell^1 := \left\{ (y_k)_{k \in \mathbb{N}}: \sum_{k=1}^{\infty} |y_k| < \infty \right\}$ denote the space of absolutely summable sequences. For an infinite matrix $M$, $M_{ij}$ denotes the $(i,j)$-entry and for $y = (y_j)_{j\in \N}$, the product $M y$ is defined by  $(M y)_i:=\sum_{j=1}^{\infty}M_{ij}y_j$ whenever the series converges. The symbol $(\cdot)^\top$ denotes the transpose of a (finite) vector. The induced operator norm of ${M}$ on $\ell^1$ is given by $\Vert M \Vert=\sup_{j\in \N}\sum_{i=1}^{\infty}\vert M_{ij}\vert$. We write $\operatorname{diag}(y_j)_{j \in \N}$ the infinite diagonal matrix with entries $y_1, y_2, ...$ along the main diagonal.

    \section{Preliminaries}\label{Sec:2}
    In this section, we first introduce the necessary concepts and recall several known results that will be used throughout the analysis. We then present the notions of admissibility and ISS for boundary control systems.

    Throughout this paper, without special statements, $X$, $Y$, and $U$ denote Banach spaces, referred to as the state space, boundary space, and control space, respectively.
	
	\subsection{A review of positive semigroups}\label{Sec:2.1}
	Here we recall basic definitions and facts about positive semigroups. For further details, we refer to \cite{BFR,Nagel}.
	
	Let $A:\bm{D}(A) \subset X \to X$ be a closed and densely defined linear operator. The \emph{resolvent set} of $A$ is defined as
	%\begin{align*}
	$\rho(A) := \{\lambda\in \mathbb{C}:\lambda \mathcal{I}-A \text{ is bijective and } (\lambda \mathcal{I}-A)^{-1} \in \mathcal{L}(X)\}$,
	%\end{align*}
	and the \emph{spectrum} is $\sigma(A) := \mathbb{C} \setminus \rho(A)$. For $\lambda \in \rho(A)$, the operator $R(\lambda, A) := (\lambda \mathcal{I}-A)^{-1}$ is called the \emph{resolvent operator} of $A$ at $\lambda$.
	The \emph{spectral bound} of $A$ is defined as $s(A) := \sup\{\operatorname{Re} \lambda : \lambda \in \sigma(A)\}$, with the convention that $s(A) = -\infty$ if $\sigma(A) = \emptyset$.
	
	A $C_0$-semigroup $T = (T(t))_{t \ge 0}$ on $X$ is a strongly continuous family of bounded linear operators satisfying $T(0) = \mathcal{I}$ and $T(t+s) = T(t)T(s)$ for all $t, s \ge 0$. Its \emph{uniform growth bound} and \emph{growth bound} are defined respectively as
	\begin{align*}
		&\omega_0(T) :=  \inf\left\{\omega \in \mathbb{R}:   \exists M \ge 1\ \text{such\ that\ } \|T(t)\| \le Me^{\omega t}, \ \forall t \ge 0\right\} , \\
		&\omega_1(T) := \inf\left\{\omega \in \mathbb{R} : \forall x\in {\bm D}(A), \exists M \ge 1 \ \text{such\ that\ }  \|T(t)x\| \le Me^{\omega t}\|x\|,\ \forall t \ge 0\right\}.
	\end{align*}
	If $A$ is the generator of $T$, then $-\infty\le s(A) \le \omega_1(T) \le \omega_0(T)< +\infty $; see, e.g., \cite{EN}. In particular, $T$ is called \emph{uniformly exponentially stable} if $\omega_0(T) < 0$, and \emph{exponentially stable} if $\omega_1(T) < 0$. If $\omega_0(T) = s(A)$, then $T$ is said to satisfy the \emph{spectrum determined growth property}; see, e.g., \cite[p.~161]{CZ}.
    %In particular, the semigroup $T$ is called uniformly exponentially stable if $\omega_0(T) < 0$, while exponentially stable if $\omega_1(T) < 0$. When $\omega_0(T)=s(A)$, the semigroup is said to satisfy the \emph{spectrum determined growth property}; see, e.g., \cite[p. 161]{CZ}.

    Now let $X$ be a Banach lattice. A $C_0$-semigroup $T$ is called \emph{positive} if $T(t)X_+ \subset X_+$ for all $t \ge 0$. The operator $A$ is called \emph{resolvent positive} if there exists $\omega \in \mathbb{R}$ such that $(\omega,\infty) \subset \rho(A)$ and $R(\lambda,A) \ge 0$ for all $\lambda > \omega$. Every generator of a positive $C_0$-semigroup is resolvent positive; however, the converse does not hold in general (see \cite{Arendt} for counterexamples and further discussion).

    For generators of positive $C_0$-semigroups on Banach lattices, the equality $s(A) = \omega_1(T)$ holds. Moreover, positive $C_0$-semigroups on $L^p$-spaces (for $1 \le p < \infty$) satisfy the spectrum determined growth property \cite{Lutz1}. Further details on the stability theory of positive semigroups can be found in \cite[Sect.~C.IV.1]{Nagel}.

   %Let $X$ be a Banach lattice. We call the $C_0$-semigroup $T$ \emph{positive} if $T(t)X_+ \subset X_+$ for all $t \ge 0$. The operator $A$ is called \emph{resolvent positive} if there exists $\omega \in \mathbb{R}$ such that $(\omega, \infty) \subset \rho(A)$ and $R(\lambda, A) \ge 0$ for all $\lambda > \omega$. While every generator of a positive $C_0$-semigroup is resolvent positive, the converse does not hold in general. Counterexamples and further discussions can be found in \cite{Arendt}. For generators of positive $C_0$-semigroups on Banach lattices, the equality $s(A)=\omega_1(T)$ holds. Moreover, positive $C_0$-semigroups on $L^p$-spaces (for $1\le p<\infty)$) satisfy the spectrum determined growth property \cite{Lutz1}. Further details on the stability theory of positive semigroups can be found in \cite[Sect. C.IV.1]{Nagel}.
	
	\subsection{Preliminaries on boundary control systems}\label{Sec:2.2}
	Consider the abstract boundary control system (see, e.g., \cite{Fattorini})
	\begin{align}\label{BC}
		\begin{cases}
			\dot{z}(t)= \tilde{A} z(t),& t>0, \quad z(0)={z_0},\\
			\Gamma z(t)=K u(t),& t>0,
		\end{cases}
	\end{align}
	where $z_0 \in X$ is the initial state, and $u \colon \R_+ \to U$ is a measurable, locally $p$-integrable function for some $p \in [1,\infty]$. The operator $\tilde{A} \colon \bm{D}(\tilde{A}) \subset X \to X$ is linear, closed, and densely defined, with domain $\bm{D}(\tilde{A})$ continuously embedded in $X$. Moreover, $\Gamma \colon \bm{D}(\tilde{A}) \to Y$ is a linear boundary operator, and $K \in \mathcal{L}(U,Y)$.
	
   In the context of PDEs, $\tilde{A}$ typically represents the spatial differential operator governing the internal dynamics, while $\Gamma$ and $K$ determine the boundary conditions of the system. The interpretation of \eqref{BC} is that the initial value problem $\dot{z}(t) = \tilde{A}z(t)$ with $z(0) = z_0$ is, in general, ill-posed and does not admit a unique solution unless the boundary condition $\Gamma z(t) = K u(t)$ is imposed. Hence, the two equations in \eqref{BC} are intrinsically coupled, with the boundary condition acting as an essential constraint that determines the evolution of the system.

  To ensure the well-posedness and stability of the boundary control system \eqref{BC}, we make the following standard assumptions:
	\begin{itemize}
		\item[{\bf (H1)}] $A:= \tilde{A}\vert_{\ker \Gamma}$ generates a $C_0$-semigroup $T:= (T(t))_{t\ge 0}$ on $X$;
		\item[{\bf (H2)}] $\Gamma$ is surjective.
	\end{itemize}
	
	Under assumptions {\bf (H1)}--{\bf (H2)}, the domain ${\bm D}(\tilde{A})$ admits the decomposition
	\begin{align}\label{decomposition}
		{\bm D}(\tilde{A})={\bm D}(A) \oplus \ker(\lambda \mathcal{I}- \tilde{A}),\qquad \lambda\in \rho(A).
	\end{align}
	Moreover, for any $\lambda \in \rho(A)$, the \emph{Dirichlet operator} associated with $(\tilde{A},\Gamma)$, defined by
	\begin{align}\label{Dirichlet}
		D_\lambda := (\Gamma|_{\ker(\lambda \mathcal{I}-\tilde{A})})^{-1} : Y \to \ker(\lambda \mathcal{I}-\tilde{A}) \subseteq X,
	\end{align}
	exists and is bounded (see \cite[Lems. 1.2 and 1.3]{Gr}). This allows us to define the \emph{boundary control operator}
	\begin{align}\label{boundary-control}
		B:=(\lambda \mathcal{I}-A_{-1})D_\lambda \in \mathcal{L}(Y,X_{-1,A}),\quad \lambda \in \rho(A),
	\end{align}
     where $X_{-1,A}$ is the extrapolation space associated with $A$ (i.e., the completion of $X$ under the norm $\|x\|_{-1,A}:=\|R(\la,A)x\|$ for $x\in  X$ and some $\lambda\in \rho(A)$), with domain $\bm{D}(A_{-1}) = X$, is the generator of $T_{-1} := (T_{-1}(t)){t \ge 0}$, the extension of $T$ to $X_{-1,A}$ (see, e.g., \cite[Sect.~II.5]{EN}).

    We note that the operator $B$ is independent of $\lambda$, satisfies $\operatorname{ran} B \cap X = \{0\}$, and yields the representation
	\begin{align}\label{representation}
		\tilde{A}x = (A_{-1} + B \Gamma)x, \qquad x \in {\bm D}(\tilde{A}),
	\end{align}
	since $\lambda D_\lambda y = \tilde{A}D_\lambda y$ for all $y \in Y$. Using this representation, the boundary control system \eqref{BC} can be reformulated as
	\begin{align}\label{distributed}
		\dot{z}(t)= A_{-1}z(t)+B K u(t),\quad t>0,\qquad z(0)=z_0.
	\end{align}
	
       For $t \ge 0$, let $\Phi_t^A \in \mathcal{L}(L^p(\R_+;Y), X_{-1,A})$ denote the \emph{input-map} of $(A,B)$, defined for each $v \in L^p(\R_+;Y)$ by
       \begin{align*}
		\Phi_{t}^A v:=\int_{0}^{t}T_{-1}(t-s)Bv(s){\rm{d}}s.
	\end{align*}
	Then, the mild solution of system \eqref{distributed} is given by
	\begin{align}\label{S2.2}
		z(t)=T(t)z_0+\Phi_{t}^A (Ku),
	\end{align}
	for all $t\ge 0$, $z_0\in X$, and $u\in L^{p}_{loc}(\R_+;U)$. Note that $z(t) \in X_{-1,A}$ for all $t \ge 0$. This motivates the following definition of admissible control operators \cite{WC}.
	\begin{definition}\label{Def.B-admi}
	 For $p \in [1, \infty]$, the operator $B$ is called an \emph{$L^p$-admissible control operator} for $A$ (or, equivalently, the pair $(A,B)$ is $L^p$-admissible) if, for some (and hence for all) $t > 0$, $\operatorname{ran} \Phi_{t}^A \subset X$.	
	\end{definition}
	
	\begin{remark}
		The ${L^p}$-admissibility property implies the following facts:
		\begin{itemize}
			\item[(i)] For every $t \ge 0$, there exists a constant $\kappa_t > 0$ such that
			\begin{align*}
				\Vert \Phi_{t}^A v \Vert \leq \kappa_t \Vert v \Vert_{L^{p}(0,t;Y)}, \quad \forall v \in L^{p}(\mathbb{R}_+;Y).
			\end{align*}
			\item[(ii)] {For any $z_0 \in X$ and any $u \in L^{p}_{loc}(\R_+;U)$ with $p \in [1,\infty)$, the solution $z(\cdot)$ of system \eqref{distributed} (and hence of \eqref{BC}) is a continuous $X$-valued function, cf. \cite[Prop. 2.3]{WC}}.
		\end{itemize}
	\end{remark}
	
	\begin{definition}
		We say that $(A,B)$ is \emph{infinite-time $L^p$-admissible} if it is $L^p$-admissible and $\kappa_\infty := \sup_{t > 0} \kappa_t < \infty$. We say that $(A,B)$ is \emph{zero-class $L^p$-admissible} if it is $L^p$-admissible and $\lim_{t \to 0} \kappa_t = 0$.
	\end{definition}
	
	\begin{remark}\label{Remark-zer0-class}
		\begin{enumerate}
			\item[(i)] Infinite-time $L^p$-admissibility implies $L^p$-admissibility. If $\omega_0(T)<0$, the two notions coincide \cite[Lem. 2.9-(i)]{JaNPS}. However, infinite-time $L^p$-admissibility does not imply uniform exponential stability of $T$ (see, e.g., \cite[Exp. 4.2.7]{TW}).
			\item[(ii)] Zero-class $L^p$-admissibility implies $L^p$-admissibility. Moreover, if $(A,B)$ is $L^p$-admissible for some $p \in [1, \infty)$, then by H\"{o}lder's inequality, $(A,B)$ is zero-class $L^q$-admissible for all $q > p$.
			\item[(iii)]  If $(A,B)$ is zero-class $L^\infty$-admissible, then the solution of system \eqref{distributed} (and hence of \eqref{BC}) satisfies $z\in {C}(\R_+;X)$ for every $z_0 \in X$ and $u \in L^{\infty}_{loc}(\R_+;U)$; see  \cite[Prop. 2.5]{JaNPS}.
		\end{enumerate}
	\end{remark}
	
	We now recall the definition of exponential input-to-state stability for boundary control systems; see \cite[Def. 2.7]{JaNPS} and \cite[Def. 3.17]{Andri}.
	\begin{definition}\label{ISS-def}
		For $p \in [1, \infty]$, system \eqref{BC} is said to be \emph{exponentially $L^p$-ISS} with linear gain if there exist constants $M, \omega, \varrho > 0$ such that for every $t \ge 0$, $z_0 \in X$, and $u \in L^p(\R_+; U)$,
		\begin{itemize}
			\item[(i)] $z  \in {C}(\R_+;X)$, and
			\item[(ii)] $\Vert z(t) \Vert \le M e^{-\omega t} \Vert z_0 \Vert + \varrho \Vert u \Vert_{L^p(\R_+; U)}$.
		\end{itemize}
        In particular, system \eqref{BC} is said to be \emph{exponentially ISS} when $p = \infty$.
   	\end{definition}
	
	\begin{remark}\label{Remark-ISS}
		According to \cite[Prop. 2.10]{JaNPS}, exponential $L^p$-ISS of system \eqref{BC} is equivalent to the uniform exponential stability of the semigroup $T$ together with the $L^p$-admissibility of $(A,B)$ for $p \in [1,\infty)$. In this case, the ISS gain satisfies $\varrho = \Vert K\Vert \kappa_\infty$, cf. \cite[Rem. 2.12]{JaNPS}. For $p=\infty$, if the state $z$ is continuous in time with respect to the norm of $X$, then the system \eqref{BC} is ISS if and only if $\omega_0(T)<0$. In particular, this holds when $(A,B)$ zero-class $L^\infty$-admissible. Further discussion can be found in \cite[Sec. 2]{JaNPS} and \cite[Sec. 3.2]{Andri}.
	\end{remark}
	
	\section{Problem statement and main results}\label{Sec:2.3}
	In this section, we establish a complete characterization of the ISS for a general class of boundary control systems. More precisely, we consider the system \eqref{BC} with $\Gamma = G-P$, where $G,P:{\bm D}(\tilde{A}) \to Y$ are linear, unbounded boundary operators that are neither closed nor closable. This unbounded, non-closable structure naturally arises from the presence of infinitely many boundary couplings, which often occur in PDE models with complex interactions at the boundary. In this setting, the boundary control problem takes the form
    	\begin{align}\label{BVCP}
		\begin{cases}
			\dot{z}(t) = \tilde{A} z(t), & t > 0, \quad z(0)=z_0, \\
			G z(t) = P z(t) + K u(t), & t > 0.
		\end{cases}
	\end{align}
	\emph{The central question we address is: Under what conditions on $\tilde{A}$, $G$, $P$, and $K$ does \eqref{BVCP} admit a unique mild solution, and how are the norms of the state $z$ and the input $u$ quantitatively related over time?} More precisely, our goal is to identify sharp criteria guaranteeing exponential $L^p$-ISS of the system.

    %and how are the norms of the state $z$ and the input $u$ related over time?} More precisely, we seek a sharp criterion for ensuring the exponential $L^p$-ISS.
	
     As noted earlier, the main difficulty in establishing the well-posedness and analyzing the asymptotic behavior of \eqref{BVCP} stems from the boundary equation, which involves the unbounded operator $P$. A well-known challenge in this context is to obtain an $X$-valued solution. Indeed, reformulating \eqref{BVCP} into a system of the form \eqref{distributed} leads to a state evolving simultaneously in two distinct extrapolation spaces of $X$, which obstructs the direct establishment of uniqueness and continuous dependence of solutions. To make this issue precise, assume that $G$ is surjective and the restriction $A:=\tilde{A}\vert_{\ker G}$ generates a $C_0$-semigroup $T$ on $X$. Following Section \ref{Sec:2.1}, we define the Dirichlet operator $D_\lambda$ and the boundary control operator $B$ as in \eqref{Dirichlet} and \eqref{boundary-control}, respectively. This yields the representation \eqref{representation} with $G$ in place of $\Gamma$. Using this representation, the boundary control problem \eqref{BVCP} can be rewritten as %the distributed parameter system
	\begin{align}\label{BVCP-distributed}
		\dot{z}(t) = \check{\calA} z(t) + \calB u(t), \quad t > 0, \qquad z(0)={z_0},
	\end{align}
	where $\check{\calA} : {\bm D}(\check{\calA}) \to X$ is the operator defined by
	\begin{align}\label{Def.checkA}
		&\check{\calA} x := (A_{-1} + B P)x, \quad
		x \in {\bm D}(\check{\calA}) :=  \{x \in {\bm D}(\tilde{A}) : (A_{-1} + B P)x \in X \},
	\end{align}
	and $\mathcal{B} : U \to X_{-1,A}$ is the control operator defined by $\mathcal{B}u := B K u$ for $u \in U$.
	
    According to \cite[Prop. 2.10]{JaNPS},  system \eqref{BVCP-distributed} is exponential $L^p$-ISS  if and only if: (i) the semigroup generated by $\check{\mathcal{A}}$ is uniformly exponentially stable, and (ii) $(\check{\mathcal{A}},\mathcal{B})$ is $L^p$-admissible; see Remark \ref{Remark-ISS}. However, a direct application of this result is prevented by the fact that $\mathcal{B}$ does not map into the extrapolation space $X_{-1,\check{\mathcal{A}}}$ associated with $\check{\mathcal{A}}$. To clarify, assume that $\check{\mathcal{A}}$ generates a C$_0$-semigroup on $X$ and let $X_{-1,\check{\mathcal{A}}}$ be the completion of $X$ with respect to the norm $\Vert x\Vert_{-1,\check{\mathcal{A}}}:=\Vert R(\la, \check{\mathcal{A}})x\Vert$ for $x\in X$. Since $B \in \mathcal{L}(Y,X_{-1,A})$ and $X_{-1,A}$ generally differs from $X_{-1,\check{\mathcal{A}}}$, the $L^p$-admissibility of $\mathcal{B}$ for $\check{\mathcal{A}}$ is not \textit{a priori} well-defined. Indeed, a necessary condition for $\mathcal{B}$ to qualify as an admissible control operator for $\check{\mathcal{A}}$ is that it belongs to $\mathcal{L}(U, X_{-1,\check{\mathcal{A}}})$. Hence, to treat $\mathcal{B}$ as a control operator for $\check{\mathcal{A}}$, we must identify a subspace of $X_{-1,\check{\mathcal{A}}}$ that can be identified with a subspace of $X_{-1,A}$ containing $\operatorname{ran} B$. An elegant approach to this identification was developed by Weiss \cite{WF} within the framework of regular linear systems. However, this framework requires verifying that an abstract quadruple of operator families satisfies certain functional equations, a task that is itself highly nontrivial; see, e.g., \cite[Thm.~6.1, Thm.~7.2, and Prop.~7.10]{WF} and \cite[Thm.~7.1.2 and Thm.~7.5.3]{Staffans}. In the sequel, we propose an alternative, purely operator-theoretic approach to overcome this difficulty when the problem is posed on Banach lattices.

    The following is the first main result of this paper, which provides sufficient conditions for the well-posedness of system \eqref{BVCP}.
	\begin{theorem}\label{Main1}
		Let $X,Y$, and $U$ be Banach lattices, and let $p \in [1,\infty]$.
		Assume the following conditions hold:
		\begin{itemize}
			\item[\bf{(A1)}] $A := \tilde{A}\vert_{\ker G}$ is a densely defined resolvent positive operator and for some $\lambda_0 > s(A)$ there exists $c > 0$ such that
			\begin{align}\label{inverse2}
				\Vert R(\lambda, A)x \Vert \geq c \Vert x \Vert, \qquad \forall \lambda \ge \lambda_0, \, x \in {X_+}.
			\end{align}
			\item[\bf{(A2)}] $G$ is surjective and $D_{\lambda} \ge 0$  {for all sufficiently large $\lambda\in \R$}.
			\item[\bf{(A3)}] $P$ is positive and $r(P D_{\lambda_1}) < 1$ for some $\lambda_1 > s(A)$.
			\item[\bf{(A4)}] For all $y \in Y$ and $y^* \in Y'$:
			$$\lim_{\R \ni \lambda \to +\infty} \langle P D_\lambda y, y^* \rangle_{Y,Y'} = 0.$$
		\end{itemize}
		Then, the  boundary control system \eqref{BVCP} is well-posed in the sense that for every initial value $z_0 \in X$ and input $u \in L^p_{loc}(\R_+;U)$, there exists a unique mild solution $z \in {C}(\R_+; X)$, and for every $\tau > 0$ there exists a constant $C_{\tau,p} > 0$, which is dependent of $p$ when $p\in(1,\infty)$ and  independent of $p$ when $p=\infty$, such that
		\begin{align*}
			\Vert z(\tau) \Vert_X  \le C_{\tau,p} \left( \Vert z_0 \Vert_X  + \Vert u \Vert_{L^p(0,\tau;U)}   \right).
		\end{align*}
	\end{theorem}
	
	\begin{remark}\label{Remark-on-well-posedness}
	 The resolvent equation implies that the family $(D_\lambda)_{\lambda>s(A)}$ is positive and monotonically decreasing. Therefore, the assumption ${D_\lambda \ge 0}$  {for all sufficiently large $\lambda\in \R$} is equivalent to $D_\lambda \ge 0$ for all $\lambda > s(A)$, and consequently $B$ is positive;  see \cite[Prop. 4.3]{SGPS} (see also \cite[Rem. 2.2]{BJVW}). Moreover, since $P$ is positive, $D_{\lambda_1}$ is positive, and $\operatorname{ran} D_{\lambda_1} \subseteq \bm{D}(\tilde{A})$ (cf.~\eqref{decomposition}), it follows from \cite[Thm.~II.5.3]{Schaf} that $P D_{\lambda_1} \in \mathcal{L}(Y)$. Therefore, the spectral radius $r(P D_{\lambda_1}) < 1$ in assumption {\bf (A3)} is well-defined.

	\end{remark}
	
	The following theorem is the second main result of this paper, providing a spectral characterization on the small-gain condition for the  exponential ISS of system \eqref{BVCP} together with an explicit ISS estimate.
	
	\begin{theorem}\label{Main2}
		Let the assumptions of Theorem \ref{Main1} be satisfied. Then, system \eqref{BVCP} is exponentially $L^p$-ISS if and only if
		\begin{align}\label{characterization}
			s(A) < 0 \quad \text{and} \quad r(P D_0) < 1.
		\end{align}
		Moreover, if $\Vert P D_0\Vert<1$, then there exist positive constants $N,a $, which  depend on the exponential decay rate of  the unforced  dynamic (i.e., system \eqref{BVCP} with $u \equiv 0$), such that
		\begin{align}\label{ISS-estimate}
			\Vert z(t,z_0,u)\Vert_X \le  N e^{-a t}\Vert z_0 \Vert_X + \frac{\Vert K \Vert \Vert D_0 \Vert}{1-\Vert P D_0 \Vert} \check{C}_p \Vert u \Vert_{L^p(\R_+; U)},
		\end{align}
		for all $t \ge 0$, $z_0 \in X$, and $u \in L^p(\R_+; U)$, where
		\begin{align*}
			\check{C}_p := \begin{cases}
				\frac{N}{c}, & p = 1 \\
				\frac{N}{c} \left( \frac{p-1}{p a} \right)^{\frac{p-1}{p}}, & p\in(1,\infty) \\
				\frac{N}{a c}  , & p=\infty
			\end{cases}
		\end{align*}
		%with $N, a > 0$ depending on the exponential decay rate of {the unforced  dynamic (i.e., system \eqref{BVCP} with $u \equiv 0$)}, and
		is  a positive constant  with $c$ given by \eqref{inverse2}.
	\end{theorem}	
	
	    \begin{remark}
        Theorem~\ref{Main2} provides a robustness result for infinite-time admissibility. Specifically, if $A$ satisfies the resolvent inverse estimate \eqref{inverse2} on the positive cone, and if the internal dynamics are positive and satisfy the small-gain condition $r(P D_0) < 1$, then the pair $(\mathcal{A},B)$ is infinite-time $L^p$-admissible. This shows that the system's admissibility property is robust under the influence of the boundary operator $P$ when the small-gain condition holds.
        \end{remark}
	
	\section{Proof of the main results}\label{Sec:3}
	This section presents detailed proofs of our main results. Let $X$, $Y$, and $U$ be Banach lattices; in particular, $X$ is equipped with its natural positive cone $X_+$. Extending the notion of positivity to the extrapolation space $X_{-1,A}$ requires careful consideration. Since $A := \tilde{A}|_{\ker G}$ satisfies assumption {\bf (A1)}, it follows from \cite[Thm. 2.5]{Arendt} that $A$ generates a positive $C_0$-semigroup $T$ on $X$. Following \cite[Def. 2.1]{BJVW}, we define the positive cone in $X_{-1,A}$, denoted by $(X_{-1,A})_+$, as the closure of $X_+$ in the norm $\|\cdot\|_{-1,A}$. This ensures that $X_+ \subset (X_{-1,A})_+$. Moreover, if $X$ is a real Banach lattice, then by \cite[Prop. 2.3]{BJVW} we have $	X_+ = X \cap (X_{-1,A})_+$. Further details and related properties can be found in \cite[Sect. 2.2]{SGPS}.
	
	\subsection{Proof of Theorem \ref{Main1}}
	To the boundary control problem \eqref{BVCP}, we associate the operator $\calA : {\bm D}(\calA) \subset X \to X$ defined by
	\begin{align}\label{def:A}
		\calA x := \tilde{A}x, \qquad x \in {\bm D}(\calA) := \left\{x \in {\bm D}(\tilde{A}) : \ Gx = Px\right\}.
	\end{align}
	Notice that $\mathcal{A}$ governs the unforced dynamics of the system \eqref{BVCP}. Under assumptions {\bf (A1)}--{\bf (A3)}, it follows from \cite[Thm. 1]{El1} that $\calA$ generates a positive semigroup $\calT:=(\calT(t))_{t\ge 0}$ on $X$, given by
	\begin{align*}%\label{variation}
		\calT(t)x = T(t)x + \int_{0}^{t} T_{-1}(t-s) B P \calT(s)x  {\rm{d}}s, \quad \forall t\ge 0, x \in {\bm D}(\calA),
	\end{align*}
	 where $B$ is the boundary control operator defined in \eqref{boundary-control}. Moreover, $\mathcal{T}$ satisfies the spectrum determined growth property. Furthermore, for any $\la>s(A)$, we have
	\begin{align*}
		\lambda>s(\calA)\iff r(PD_\la)<1,
	\end{align*}
    and the resolvent operator of $\mathcal{A}$ can be expressed as
 	\begin{align}\label{resolvent-operator}
		R(\lambda, \calA) = (\mathcal{I}-D_\lambda P)^{-1} R(\lambda, A)
		= R(\lambda, A) + D_\lambda (\mathcal{I}-P D_\lambda)^{-1} P R(\lambda, A).
	\end{align}
	According to \cite[Prop.~2.5]{JaNPS}, to prove Theorem~\ref{Main1}, it suffices to show that $\mathcal{A}$ coincides with $\check{\mathcal{A}}$ defined in \eqref{Def.checkA}, and that the pair $(\mathcal{A}, \mathcal{B})$ is zero-class $L^p$-admissible. To this end, we first identify a subspace of $X_{-1,\mathcal{A}}$ with a subspace of $X_{-1,A}$ containing $\operatorname{ran} B$. Inspired by \cite[Sect.~7]{WF}, we define the operator $J^A: {\bm D}(J^A) \subset X_{-1,A} \to X_{-1,\mathcal{A}}$ by
\begin{subequations}
	\begin{align}
		 J^Ax:=&\lim_{\la \to +\infty} \Vert \la R(\la,A)x\Vert_{-1,\calA},\label{extension-identity1}\\
		 {\bm D}(J^A):=&\{x\in X_{-1,A}:\text{ the limit in } \eqref{extension-identity1}\; {\rm exists}\}\label{extension-identity2}.
	\end{align}
\end{subequations}
	Notice that $J^A$ is an extension of the identity on $X$:
	\begin{align}\label{Tool}
		J^Ax=x,\qquad \forall x\in X.
	\end{align}
	%In what follows, let $X_1$ denote ${\bm D}(A)$ endowed with the graph norm, and let $Z$ denote ${\bm D}(\tilde{A})$ endowed with the norm
    Next, let $X_1 := {\bm D}(A)$ with the graph norm, and define the Banach space $Z := {\bm D}(\tilde{A})$ equipped with the norm
	\begin{align*}
		\Vert z\Vert_{Z}^2:=\inf\left\{\Vert x\Vert^2+\Vert y\Vert^2:
		x\in {\bm D}(A),y\in Y,z=x+R(\la,A_{-1}) B y \right\}.
	\end{align*}
	Then, $X_1\subset Z\subset X$ with continuous embeddings.  We note that the space $Z$ and its topology are independent of the choice of $\la$. Moreover, we have the following result.
	\begin{lemma}\label{P-regularity}
		Let $A := \tilde{A}|_{\ker G}$ be the generator of a positive $C_0$-semigroup $T$ on $X$ and $P \ge 0$. Suppose assumptions {\bf (A2)} and {\bf (A4)} are satisfied. Then
		\begin{align}\label{Limit-Px}
			\lim_{\la \to +\infty} \Vert \la P R(\la,A)x-Px\Vert_Y=0, \qquad \forall x\in Z.
		\end{align}
	\end{lemma}
	\begin{proof}
		For $n \in \N$ large enough ($n>\omega_0(T)$), let us introduce the Yosida-like approximations $P_n \in \mathcal{L}(X,Y)$ defined by
		\begin{align*}
			P_n x:=P n R(n,A)x,\qquad x\in X.
		\end{align*}
		Let $y\in Y_+$. Since $A$ is a resolvent positive operator and $P,D_\la$ are positive operators, it follows from the resolvent equation and assumption {\bf (A4)} that the sequence $(P D_{n}y)$ decreases monotonically and weakly converges to $0$. Applying Dini's theorem for Banach lattices (see, e.g., \cite[Prop. 10.9]{BFR}), we obtain
		\begin{align}\label{regularity}
			\lim_{n\to +\infty } \Vert P D_{n}y\Vert_Y=0, \quad \forall y\in Y_+,
		\end{align}
		and hence for all $y\in Y$.
        %Now recall that
	%	\begin{align}\label{Eq1}
	%		&\lim_{n\to +\infty} \Vert n R(n,A)x-x\Vert_{X_1}=0,\qquad \forall x\in X_1.%\label{Eq1}\\
			%& \lim_{n\to +\infty} \Vert n R(n,A)x-x\Vert_{X}=0,\qquad\; \forall x\in X,\label{Eq2}
	%	\end{align}
        %see, e.g., \cite[Lem. II.3.4]{EN}.
        Using the identity
		\begin{align*}
			nP R(n,A) D_\la =\frac{1}{1-\frac{\la}{n}}(P D_\la-P D_n),\qquad (\rho(A)\ni\la\neq n),
		\end{align*}
		it follows from \eqref{regularity} that
		\begin{align}\label{nGamma}
			\lim_{n \to +\infty} \Vert n P R(n,A) D_\la y-P D_\la y \Vert_{Y}= 0, \quad \forall y\in Y.
		\end{align}
		Now, for any $z\in Z$, we know from \eqref{decomposition} that there exist $x\in {\bm D}(A)$ and $y\in Y$ such that $z=x+D_{\la}y$ for some $\lambda\in\rho(A)$. Therefore,
		\begin{align*}
			\Vert P_n z-P z\Vert_{Y}%&\le \Vert P_n x-P x\Vert_{Y}+\Vert P_n D_{\la} y-P D_\la y\Vert_{Y}\\
			%&\le \Vert P\Vert_{\calL(Z, Y)} \Vert n R(n,A)x- x\Vert_{Z}\\&+ \Vert n P R(n,A) D_\la y-P D_\la y\Vert_{Y}\\
			&\le \Vert P\Vert_{\calL(Z,Y)} \Vert n R(n,A)x- x\Vert_{X_1}+\Vert n P R(n,A) D_\la y-P D_\la y\Vert_{Y}.
		\end{align*}
		Thus, passing to the limit as $n\to +\infty$ and taking into account \eqref{nGamma}, we obtain \eqref{Limit-Px}. For more details we refer to step 1 in the proof of \cite[Thm. 2]{Yassine}.
	\end{proof}

	Let $Z_{-1,A}$ denote the completion of $Z$ with respect to the norm
	\begin{align}\label{space-Z}
		\Vert x\Vert_{Z_{-1,A}}:=\Vert R(s,A)x\Vert_{Z},\qquad x\in Z,
	\end{align}
	for some $s\in \rho(A)$. Then, $X \subset Z_{-1,A} \subset X_{-1,A}$  with continuous embedding. Similarly, we define $Z_{-1,\calA}$ as the analogue of $Z_{-1,A}$ for $\calA$, equipped with the norm $\Vert \cdot\Vert_{Z_{-1,\calA}}$ given by a formula analogous to \eqref{space-Z}. We have $X \subset Z_{-1,\calA} \subset X_{-1,\calA}$ with continuous embeddings.
	
	We have the following result.
	\begin{lemma}\label{technical-lem1}
		Consider the setting of Theorem \ref{Main1}. Then, $Z_{-1,A}\subset {\bm D}(J^A)$ and
		\begin{align}\label{boundness-J}
			J^A\in \calL(Z_{-1,A},Z_{-1,\calA}).
		\end{align}
	\end{lemma}
	
	\begin{proof}
		Let $z\in Z_{-1,A}$. According to \eqref{extension-identity2}, we need to show that the limit of $\la R(\la,A)z$ exists in $X_{-1,\calA}$ as $\la \to +\infty$. Take $ s>s(\calA)$. Using \eqref{resolvent-operator}, one has
		\begin{align*}
			R(s,\calA)\la R(\la,A)z&=(\mathcal{I}-D_s P)^{-1}R(s,A)\la R(\la,A)z\\
			&=R(s,A)\la R(\la,A)z+D_s (\mathcal{I}-P D_s )^{-1}P R(s,A)\la R(\la,A)z.
		\end{align*}
		Then, for $x=R(s,A)z$, we obtain
		\begin{align*}
			\lim_{\la \to +\infty}R(s,\calA)\la R(\la,A)z=\lim_{\la \to +\infty} \la R(\la,A)x
			+D_s (\mathcal{I}-P D_s )^{-1}\lim_{\la \to +\infty} P \la R(\la,A) x.
		\end{align*}
		The first limit on the right-hand side is $R(s,A)z$ (by \eqref{Tool}), while the second limit exists due to Lemma \ref{P-regularity} as $R(s,A)z\in Z$. Thus, $z\in {\bm D}(J^A)$, and
		\begin{align}\label{J-expression}
			R(s, \calA)J^Az=R(s,A)z+D_s (\mathcal{I}-P D_s )^{-1}P R(s,A)z.
		\end{align}
		Moreover, form \eqref{space-Z} and \eqref{J-expression}, we deduce
		\begin{align*}
			\Vert J^Az\Vert_{Z_{-1,\calA}}&=\Vert R(s, \calA)J^Az\Vert_{Z}\\
			&\le \Vert R(s,A)z\Vert_{Z}+ \Vert D_s (\mathcal{I}-P D_s )^{-1}P R(s,A)z\Vert_{Z}\\
			&\le \left(1+\Vert D_s(\mathcal{I}-P D_s )^{-1}P\Vert_{\calL(Z,Y)}\right)\Vert z\Vert_{Z_{-1,A}},
		\end{align*}
		for all $z\in Z_{-1,A}$. This proves $J^A\in \calL(Z_{-1,A},Z_{-1,\calA})$. The proof is complete.
	\end{proof}

	We now define the operator $J^\calA$ analogously to the definition of $J^A$ in \eqref{extension-identity1}-\eqref{extension-identity2}. More precisely, we have
%%\begin{subequations}
	\begin{align}
		 J^\calA x:=&\lim_{\la \to +\infty} \Vert \la R(\la,\calA)x\Vert_{-1,A},\label{definition-J'}\\
		  {\bm D}(J^\calA):=&\{x\in X_{-1,\calA}: \text{ the limit in \eqref{definition-J'}}\text{ exists in } X_{-1,A}\}. \notag
	\end{align}
%\end{subequations}
	Clearly, $J^\calA$ satisfies
	\begin{equation}\label{extension-identityJ'}
		J^\calA x=x,\qquad \forall x\in X.
	\end{equation}
	
In addition, the following statement holds true.
	\begin{lemma}\label{technical-lem2}
		Consider the setting of Theorem \ref{Main1}. Then, $Z_{-1,\calA}\subset {\bm D}(J^\calA)$ and
		\begin{align}\label{boundness-J'}
			J^\calA \in \calL(Z_{-1,\calA},Z_{-1,A}).
		\end{align}
	\end{lemma}
	\begin{proof}
		First, we establish that for all $x\in {\bm D}(\tilde{A})$,
		\begin{align}\label{regularity-P-wrt-calA}
			\lim_{\la \to +\infty}\la PR(\la,\calA)x=Px.
		\end{align}
		Let $x\in {\bm D}(\tilde{A})$. Using \eqref{resolvent-operator},  Lemma \ref{P-regularity}, and assumption {\bf (A4)}, we obtain
		\begin{equation*}
			\lim_{\la \to +\infty}\la PR(\la,\calA)x =\lim_{\la \to +\infty}\la PR(\la,A)x+\lim_{\la \to +\infty}  PD_\la (\mathcal{I}-PD_\la )^{-1}\la PR(\la,A)x
			=Px.
		\end{equation*}
		Now, let $z\in Z_{-1,\calA}$. By \eqref{definition-J'}, we must show that $\la R(\la,\calA)z$ converges in $X_{-1,A}$ as $\la \to +\infty$. Fix $ s>s(\calA)$. Applying \eqref{resolvent-operator}, we obtain
		\begin{align*}
			R(s,A)\la R(\la,\calA)z&=(R(s,\calA)-D_s PR(s,\calA))\la R(\la,\calA) z\\
			&=\la R(\la,\calA) R(s,\calA)z-D_s P\la R(\la,\calA)R(s,\calA)z.
		\end{align*}
		Taking the limit $\lambda \to +\infty$ and invoking \eqref{extension-identityJ'} and \eqref{regularity-P-wrt-calA}, we obtain
		\begin{align*}
			\lim_{\la \to +\infty}R(s,A)\la R(\la,\calA)z&=\lim_{\la \to +\infty}\la R(\la,\calA) R(s,\calA)z-D_s \lim_{\la \to +\infty}\la P R(\la,A)R(s,\calA)z\\
			&=R(s,\calA)z-D_s PR(s,\calA)z.
		\end{align*}
		This proves $Z_{-1,\calA}\subset {\bm D}(J^\calA)$  with
		\begin{align*}
			R(s,A)J^\calA z=R(s,\calA)z-D_s PR(s,\calA)z.
		\end{align*}
		From the above expression, we derive the following norm estimate
		%{which implies that}
		\begin{align*}
			\Vert J^\calA z\Vert_{Z_{-1,A}}=\Vert R(s, A)J^\calA z\Vert_{Z}
			\le \left(1+\Vert D_s P\Vert_{\calL(Z,Y)}\right)\Vert z\Vert_{Z_{-1,\calA}},
		\end{align*}
		for all $z\in Z_{-1,\calA}$. Thus, $J^\calA \in \calL(Z_{-1,\calA},Z_{-1,A})$. This ends the proof.
	\end{proof}
	
	In what follows, we define the Banach space $W_1$ as the closure of ${\bm D}(A)$ in $Z$. Then we have dense and continuous embeddings $X_1\subset W_1 \subset X$. Next, let $W_{-1,A}$ denote the closure of $X$ in $Z_{-1,A}$. For any $s\in \rho(A)$, it holds that
	\begin{align}\label{definition-W}
		W_{-1,A}=(s\mathcal{I}-A)W_{1}.
	\end{align}
	Similarly, we define $W_{-1,\calA}$ as the analogue of $W_{-1,A}$ for $\calA$; that is, $W_{-1,\calA}$ is the closure of $X$ in $Z_{-1,\calA}$. With these definitions, we have dense and continuous embeddings
	\begin{align*}
		X\subset W_{-1,A} \subset X_{-1,A},\qquad X\subset W_{-1,\calA} \subset X_{-1,\calA}.
	\end{align*}
	%densely and with continuous embeddings. %For further details on this construction, we refer to \cite[Sect. 7]{WR}.
	
	The spaces introduced above possess the following fundamental properties.
	\begin{lemma}\label{technical-lem3}
		Consider the setting of Theorem \ref{Main1}. Then $J^A$ is an isomorphism from $W_{-1,A}$ to $W_{-1,\calA}$ and $(J^A)^{-1}=J^\calA$.
	\end{lemma}
	\begin{proof}
		The proof follows by adapting the arguments in \cite[Thm. 7.7]{WF}, using Lemmas \ref{technical-lem1} and \ref{technical-lem2}. The technical details are analogous and are omitted for brevity.
	\end{proof}
	
	Using this isomorphism, we identify $W_{-1,A}$ and $W_{-1,\calA}$, so that we make no distinction between $x\in W_{-1,A}$ and its image $J^A x\in W_{-1,\calA}$. Moreover, for any $x\in W_{-1,A}$,
	\begin{align*}
		x=\lim\limits_{\la\to +\infty} \lambda R(\lambda,A_{-1})x=\lim\limits_{\la\to +\infty} \lambda R(\lambda,\calA_{-1})x\quad \text{in both } X_{-1,A} \text{ and } X_{-1, \calA}.
        \end{align*}
	
     This identification allows us to prove the following result.
     \begin{lemma}\label{echnical-lem4}
		Consider the setting of Theorem \ref{Main1}. Then,
		\begin{align}\label{Boundness-of-B}
			B\in \calL(Y,W_{-1,\calA}).
		\end{align}
	\end{lemma}
	\begin{proof}
		By the closed graph theorem, $D_\la \in \calL(Y, Z)$ for any $\la \in \rho(A)$, since $\operatorname{ran} D_\la\subset Z$ and $Z\subset X$ with continuous embedding. Therefore, $D_\la \in \calL(Y, W_1)$, and by \eqref{definition-W} it follows that $B\in \calL(Y,W_{-1,A})$.  Hence, according to Lemma \ref{technical-lem3}, $B\in \calL(Y,W_{-1,\calA})$. This completes the proof.
	\end{proof}
	
   With these preparations in place, we are now ready to give the rigorous proof of Theorem \ref{Main1}.

    \emph{Proof of Theorem \ref{Main1}:\enspace}
		Since $G$ is surjective and $A := \tilde{A}|_{\ker G}$ generates a positive $C_0$-semigroup, it follows from Section \ref{Sec:2.2} that the boundary control system \eqref{BVCP} can be reformulated as  system \eqref{BVCP-distributed}. To establish the well-posedness of \eqref{BVCP-distributed}, we proceed in two steps.
		
		First, we show that the operator $\check{\calA}$, defined in \eqref{Def.checkA}, generates a $C_0$-semigroup on $X$. In fact, using the representation \eqref{representation} with $G$ in place of $\Gamma$, we obtain
		\begin{align*}
			\calA x = \tilde{A} x &= (A_{-1} + B G)x
			= (A_{-1} + B P)x
			= \check{\calA} x,
		\end{align*}
		for all $x \in {\bm D}(\calA)$, which shows that ${\bm D}(\calA) \subset {\bm D}(\check{\calA})$. For the reverse inclusion, let $\lambda > s(A)$ and $x \in {\bm D}(\check{\calA})$. Then,
		\begin{align*}
			R(\lambda, A) \check{\calA} x
			= R(\lambda, A)(A_{-1} + B P)x
			= \lambda R(\lambda, A)x + (D_\lambda P x-x).
		\end{align*}
		Applying $G$ to both sides yields $G(D_\lambda P x-x) = 0$, which implies $Px=Gx$ and therefore $x \in {\bm D}(\calA)$. This proves that $\check{\calA} = \calA$, and consequently the operator $\check{\calA}$ generates a positive $C_0$-semigroup $\calT$ on $X$. %Furthermore, we have $B \in \calL(Y, W_{-1, \calA})$ due to Lemma \ref{echnical-lem4}.
        Thus,  system \eqref{BVCP-distributed}--and hence also \eqref{BVCP}--is equivalent to the following system
		\begin{align}\label{Cauchy-problem}
			\dot{z}(t) = \calA_{-1} z(t) + \calB u(t), \quad t > 0,\qquad z(0)={z_0},
		\end{align}
		where $\calB:=BK\in \calL(U, W_{-1, \calA})$ due to Lemma \ref{echnical-lem4}.
		
		We now establish that $(\calA,\calB)$ is zero-class $L^p$-admissible. Using assumption {\bf (A3)} and noting that the family $(P D_\lambda)_{\lambda > s(A)}$ is positive and monotonically decreasing, it follows from \eqref{resolvent-operator} that
		\begin{align}\label{Resolvent}
			R(\lambda,\calA)
			= R(\lambda,A)+D_\lambda \sum_{n=0}^{\infty} (P D_\lambda)^n P R(\lambda, A)
			\ge R(\lambda, A) \ge 0,
		\end{align}
		for all $\lambda \ge \lambda_1$. By the monotonicity of the norm, we then obtain $\|R(\lambda, \calA)x\| \ge \|R(\lambda, A)x\|$ for all
		$\lambda \ge \lambda_1$ and $x \in X_+$. Thus, for any $\lambda > \max\{\lambda_0, \lambda_1\}$ there exists $c := c(\lambda) > 0$ such that
		\begin{align}\label{Inverse-F}
			\|R(\lambda, \calA)x\| \ge c \|x\|, \qquad \forall\, x \in X_+.
		\end{align}
		Therefore, by \cite[Thm. 2.7]{ELLC}, the pair $(\calA, B)$ is $L^1$-admissible. Since $K$ is bounded, it follows from Remark \ref{Remark-zer0-class}-(c) that the pair $(\calA, \calB)$ is zero-class $L^p$-admissible for any $p\in (1,\infty]$. Hence, according to \cite[Prop. 2.5]{JaNPS}, we conclude that system \eqref{BVCP} admits a unique mild solution $z \in  {C}(\mathbb{R}_+; X)$ given by
		\begin{align}\label{variation1}
			z(t) = \calT(t)z_0 + \int_{0}^{t} \calT_{-1}(t-s) \calB u(s) {\rm{d}}s,
		\end{align}
		for all $z_0 \in X$ and $u \in L_{loc}^p(\mathbb{R}_+; U)$. This completes the proof. \hfill $\blacksquare$\par

	\subsection{Proof of Theorem \ref{Main2}}
	Before proceeding with the proof, we first derive an infinite-time $L^p$-admissibility result.
	
	\begin{lemma}\label{Lem-admi}
		%Let $G$ be surjective and
		Let $p \in [1, \infty]$. Suppose that the restriction $A$ is a densely defined resolvent positive operator such that
		\begin{align}\label{inverse-estimate}
			\Vert R(\lambda_0, A)x \Vert \geq c \Vert x \Vert, \quad \forall x \in X_+,
		\end{align}
		for some $\lambda_0 > s(A)$ and $c > 0$. If $R(\lambda,A_{-1})B$ is positive for all sufficiently large $\lambda \in \R$ and $s(A) < 0$, then there is a constant $\check{C}_p>0$ such that
		\begin{align*}
			\left\Vert \Phi_t^A v \right\Vert \le \check{C}_p \left\Vert R(\lambda_0, A) B \right\Vert \Vert v \Vert_{L^p(\R_+;Y)},\quad \forall t \ge 0,v \in L^p(\R_+; Y).
		\end{align*}
		%for all $t \ge 0$ and $v \in L^p(\R_+; Y)$,
		%where {$\check{C}_p$ is a positive constant given by \eqref{check-p}.}
		%	%\begin{align*}
		%		$\check{C}_p := \begin{cases}
			%			\frac{M}{c}, & p = 1,\\
			%			\frac{M}{c} \left( \frac{p-1}{p \omega} \right)^{\frac{p-1}{p}}, & p \in(1,\infty), \\
			%			\frac{M}{c\omega}, & p =\infty,
			%		\end{cases}$
		%	%\end{align*}
		%	for constants $M, \omega > 0$.
		Moreover, $(A,B)$ is infinite-time $L^p$-admissible.
	\end{lemma}
	
	\begin{proof}
		By \cite[Thm. 2.5]{Arendt}, the operator $A$ generates a positive $C_0$-semigroup $T$ on $X$ such that $\omega_0(T) = s(A)$. Moreover, by \cite[Prop. 4.3]{SGPS}, the boundary control operator $B$ is positive. Thus, according to \cite[Thm. 2.7]{ELLC}, $\Phi_t^A v \in X_+$ for all $t \ge 0$ and $v \in L^p(\R_+;Y)$.
		
		For $s(A)<0$, there exist ${\xi},\omega>0$ such that $\Vert T(t)\Vert \leq {\xi}e^{-\omega t}$ for all $t\ge 0$. Using  {\eqref{inverse-estimate}} and H\"{o}lder's inequality, we obtain
		\begin{align*}
			c\Vert \Phi_{t}^A v\Vert & \le \left\Vert R(\lambda_0,A) \Phi_{t}^A v\right\Vert \\
			&= \left\Vert \int_{0}^{t}T(t-s)R(\lambda_0,A_{-1})B v(s){\rm{d}}s\right\Vert \\
			&\leq {\xi}\Vert R(\lambda_0,A_{-1})B\Vert\left(\int_{0}^{t}e^{-\frac{\omega p}{p-1} (t-s)}{\rm{d}}s\right)^{\frac{p-1}{p}} \left( \int_{0}^{t} \Vert v(s) \Vert^p {\rm{d}}s\right)^{\frac{1}{p}},
		\end{align*}
		for all $v\in {L^p_+}(\R_+;Y)$. Computing the integral yields
		\begin{align}\label{estimate-input-map}
			\left\Vert \Phi_t^A v \right\Vert \le \check{C}_p \left\Vert R(\lambda_0, A_{-1}) B \right\Vert \Vert v \Vert_{L^p(\R_+;Y)}, \quad \forall v \in L^p_+ (\R_+; Y),
		\end{align}
		where
		\begin{align*}%\label{check-p}
			\check{C}_p := \begin{cases}
				\frac{{\xi}}{c}, & p = 1,\\
				\frac{{\xi}}{c} \left( \frac{p-1}{p \omega} \right)^{\frac{p-1}{p}}, & p \in(1,\infty), \\
				\frac{{\xi}}{c\omega}, & p =\infty.
			\end{cases}
		\end{align*}
		Since $\Phi_t^A$ is positive for all $t \ge 0$, then $\left\vert \Phi_t^A v \right\vert \le \Phi_t^A \vert v \vert$ for all $v \in L^p(\R_+; Y)$. Therefore, the estimate \eqref{estimate-input-map} holds for any $v \in L^p(\R_+; Y)$, which establishes the infinite-time $L^p$-admissibility of $(A,B)$.
	\end{proof}

    We can now present the proof of Theorem \ref{Main2}.

 	 \emph{Proof of Theorem \ref{Main2}: \enspace} 		% Since $(\calA, \calB)$ is {$L^1$-admissible} and {zero-class $L^p$-admissible for any $p\in (1,\infty]$}
		From the proof of Theorem \ref{Main1}, system \eqref{BVCP} is equivalent to system \eqref{Cauchy-problem}. Thus, according to \cite[Prop. 2.10]{JaNPS},   system \eqref{BVCP} is exponentially $L^p$-ISS if and only if $\omega_0(\calT)<0$. Using \cite[Thm. 2]{El1}, this is equivalent to the spectral condition \eqref{characterization}.
		
      Next, we establish the ISS estimate \eqref{ISS-estimate}. Assume that the condition \eqref{characterization} holds. It follows from \cite[Thm. 2]{El1} that there exist $N, a > 0$ such that
		\begin{align}\label{Expon-sta}
			\|\calT(t)\| \le N e^{-a t}, \qquad \forall\, t \ge 0.
		\end{align}
		For $t \ge 0$, denote by $\Phi_t^{\calA}$ the input-map of $(\calA, \calB)$ defined by
		\begin{align*}
			\Phi_t^{\calA} u := \int_0^t \calT_{-1}(t-s)\, \calB u(s) {\rm{d}}s, \qquad
			\forall u \in L^p(\mathbb{R}_+; U).
		\end{align*}
		Since $\calA$ satisfies the resolvent condition \eqref{Inverse-F}, it follows from Lemma \ref{Lem-admi} that
		\begin{align*}
			\|\Phi_t^{\calA} u\| \le \frac{N}{c} \left\| R(\lambda, \calA_{-1}) \calB \right\|
			\check{C}_p\|u\|_{L^p(\R_+; U)},
		\end{align*}
		for all $t \ge 0$, $u \in L^p(\R_+; U)$, and some $\lambda > \max\{\lambda_0,\lambda_1\}$, where $\lambda_0$ and $\lambda_1$ are as in assumptions {\bf (A1)} and {\bf (A3)}. By the positivity and monotonicity of the resolvent map $\lambda \mapsto R(\lambda, \calA)$ on $(s(\calA),\infty)$ and the condition $s(\calA)<0$, we have
		\begin{align*}
			0\le R(\lambda, \calA) \le R(0, \calA), \qquad \la \ge 0.
		\end{align*}
		Thus, for all $\lambda \ge 0$ and $u \in U$,
		\begin{align*}
			\left\vert R(\lambda, \calA_{-1}) B K u \right\vert \le  \left\vert R(0, \calA_{-1}) B K u \right\vert \le  R(0, \calA_{-1}) B  \left\vert K u\right\vert,
		\end{align*}
		due to the facts that $B\in \calL(U, W_{-1, \calA})$ and $R(0, \calA_{-1}) B$ is positive. The lattice property of the norms on $X$, $Y$, and $U$ then implies
		\begin{align*}
			\left\Vert R(\la,\calA_{-1})\calB\right\Vert \le \Vert R(0,\calA_{-1})B \Vert \Vert K\Vert.
		\end{align*}
		On the other hand, using \eqref{resolvent-operator}, one obtains that
		\begin{align*}
			R(0,\calA_{-1}) \calB u  = D_0 (\mathcal{I}-P D_0)^{-1} Ku, \qquad \forall u\in U.
		\end{align*}
		For $\|P D_0\| < 1$, the Neumann series expansion yields
		\begin{align*}
			\Vert R(0, \calA_{-1}) \calB \Vert \le \Vert D_0\Vert \Vert K\Vert  \frac{1}{1 - \Vert P D_0\Vert}.
		\end{align*}
		Hence, for $\la>0{\vee (\la_0 \vee \la_1)}$, we get
		\begin{align}\label{Phi-estimate}
			\Vert\Phi_t^{\calA} u\Vert
			\le  \Vert D_0\Vert \Vert K\Vert  \frac{1}{1- \Vert P D_0\Vert}
			\check{C}_p \|u\|_{L^p(\R_+; U)},\quad  \forall t \ge 0,\ u \in L^p(\R_+;U).
		\end{align}
		The ISS estimate \eqref{ISS-estimate} now follows by combining \eqref{Expon-sta} with \eqref{Phi-estimate} through the variation of constants formula \eqref{variation1}. This ends the proof. \hfill $\blacksquare$\par

  %    \begin{remark}
%      Compared with Lemma \ref{Lem-admi}, Theorem \ref{Main2} establishes a robustness result for infinite?time admissibility. More precisely, if $A$ satisfies the resolvent inverse estimate \eqref{inverse2} on the positive real axis, and if the internal dynamics are positive and satisfy the small-gain condition $r(PD_0) < 1$, then $(\mathcal{A},B)$ is infinite-time admissible.
%      \end{remark}
%	
	
	\section{Application:  kinetic transport networks}\label{Sec:4}%governed by the Boltzmann-like equations}
	As an application of the developed theory, we consider the flow of material through an infinite network of interconnected circular edges (see Fig.~\ref{FigN}).
 On each circle $j\in \mathcal{J}$ (a countably infinite index set), the particle distribution is described by a kinetic density function $z_j(t,x,v)$, where $t \ge 0$ represents time, $x \in [0, l_j]$ denotes the position along the circle of length $l_j\in (0,\infty)$, and $v \in [v_{\min}, v_{\max}]$ is particle velocity with fixed bounds $0 < v_{\min} \le v_{\max} < \infty$.
		The evolution on the phase space $\Omega_j := (0, l_j) \times [v_{\min}, v_{\max}]$ is governed by the following linear  Boltzmann-like  equations
	\begin{align}\label{Eqt1}
		\begin{cases}
			\dot{z}_{j}(t,x,v) = -v\partial_xz_{j}(t,x,v)-q_{j}(x,v)z_{j}(t,x,v), & t> 0,\ (x,v)\in \Omega_j,\\
			z_{j}(0,x,v) = f_{j}(x,v), & (x,v)\in \Omega_j,
		\end{cases}
	\end{align}
	for $j\in \mathcal{J}$. Here, the advection term $-v\partial_x z_j$ models particle transport along the edge toward the junction at $x = 0$, {where $\partial_x$ denotes the partial derivative with respect to $x$ in the sense of distributions}. The coefficient $q_j(x, v)$ governs local mass absorption or generation, while $f_j(x,v)$ specifies the initial condition.
	
	The interaction at the junction O is prescribed by the TCs
	  \begin{align}\label{Eqt2}
			\begin{cases}
				v z_{i}(t,0,v) = \displaystyle\sum_{j\in \mathcal{J}} w_{ij}\left[\displaystyle\int_{-r_j}^{0}{\rm d}{\eta_j(\theta)} \int_{v_{\min}}^{v_{\max}}{\beta_j}(v,v')v'z_{j}(t+\theta,l_j,v') {\rm d}v+ u(t,v)\right], & \\
				z_{j}(\theta,l_j,v) = {\varphi_j}(\theta,v), &
			\end{cases}
	\end{align} for $i,j\in \mathcal{J}$, $t\ge 0$, and $(\theta,v)\in \Theta_j := [-r_j, 0] \times [v_{\min},v_{\max}]$, with finite time delays $r_j \in (0,\infty)$. The left-hand side of \eqref{Eqt2} represents the outgoing flux from the junction into circle $i$. The right-hand side describes a superposition of historical scattered fluxes: particles that arrived at the junction from all incoming circles $j$ during the time interval $[t-r_j,t]$ first undergo velocity redistribution through the scattering kernel $\beta_j(v,v')$. The scattered fluxes are then weighted by the delay measure ${\rm d}\eta_j(\theta)$ before being routed into outgoing circles according to the coefficients $w_{ij}\in (0,\infty)$. The term $u(t,v)$ models an external disturbance, i.e., uncontrolled flux perturbations. The second equation in \eqref{Eqt2} supplies the required historical initial condition on the delay interval, given by $\varphi_j(\theta,v)$.
	
	\begin{figure}[htb]
		\centering
		\includegraphics[width=6cm, height=5cm]{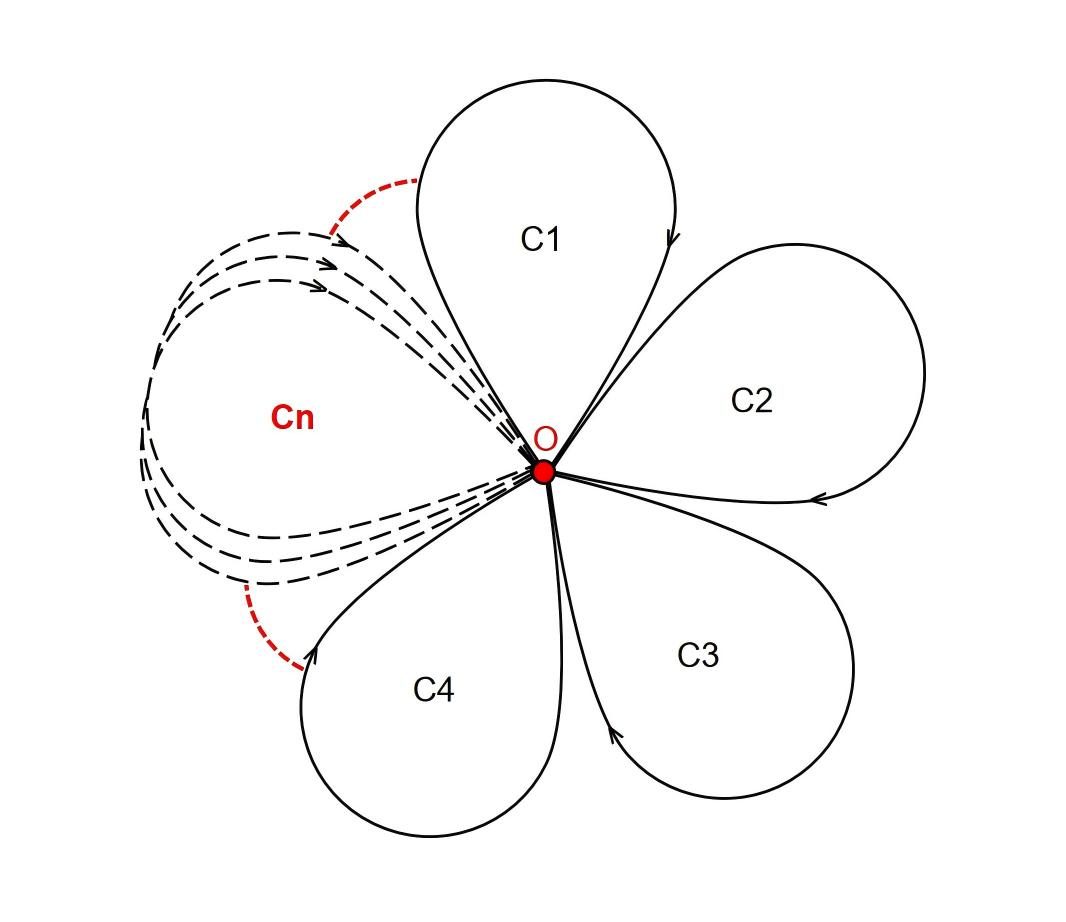}
		\caption{{An infinite network of intersecting circles connected at the junction O.}}\label{FigN}
		%\normalsize
	\end{figure}
	%\vspace{.1cm}
	
	With respect to the definition of the terms we have introduced in \eqref{Eqt1}-\eqref{Eqt2}, in order to consider the kinetic transport network system, we propose the following assumptions.
	\begin{mainassumptions}\label{Mainassump}
		\begin{itemize}
			\item[(i)] $ q_j\in L^{\infty}(\Omega_j) $ for all $j\in \mathcal{J}$ and there exist  constants $\gamma_1,\gamma_2\in \R$ such that
			\begin{align}\label{gamma}
				\gamma_1\le	q_j(x,v)\le \gamma_2 , \quad\forall j\in \mathcal{J}, \ \text{a.e.}\ (x,v)\in \Omega_j;
			\end{align}
			\item[(ii)] $\beta_j\in {L^{\infty}_+}([v_{\min},v_{\max}]^2)$ for all $j\in \mathcal{J}$.
		\end{itemize}
	\end{mainassumptions}
	
	In what follows, let $V := L^1(v_{\min}, v_{\max})$ denote the velocity distribution space. In addition to Assumption \ref{Mainassump}, we impose the following conditions on the functions $\eta_j$ appearing in \eqref{Eqt2}.
	
	\begin{mainassumptions}\label{Mainassump1}
		For each $j \in \mathcal{J}$, let $\eta_j: [-r_j, 0] \to \mathcal{L}{(V)}$ be an operator-valued function satisfying:
		\begin{itemize}
			\item[(i)]  $\eta_j$ is of bounded variation, i.e.,
			\begin{align*}
				{\rm Var}(\eta_j;-r_j,0) :=\sup_{\underset{m\in \N}{\theta_1=-r_j<\ldots<\theta_m=0}} \sum_{k=1}^{m}\Vert \eta_j(\theta_k)-\eta_j(\theta_{k-1})\Vert_{\calL(V)}<+\infty.
			\end{align*}
			\item[(ii)] $\eta_j$ is continuous from the left, i.e., $				\lim_{\theta \mapsto \theta_0^{-}}\eta_j(\theta)=\eta_j(\theta_0^{-})$.
			
		\end{itemize}
	\end{mainassumptions}
	
	To apply our abstract results, we first reformulate the kinetic transport network system \eqref{Eqt1}-\eqref{Eqt2} in suitable path spaces. Define the edge state-space
	\begin{align*}
		X^{\mathsf{e}}:=\big\{ (f_j)_{j\in \mathcal{J}} : f_j \in L^1(0,l_j;V),
		\sum_{j\in \mathcal{J}} \|f_j\|_{L^1(0,l_j;V)} < \infty \big\},
	\end{align*}
	endowed with the norm $\|f \|_{X^{\mathsf{e}}} :=\sum_{j\in \mathcal{J}} \|f_j\|_{L^1(0,l_j;V)}$. The space $X^{\mathsf{e}}$ represents the space of all possible particle density distributions along the network, where the $L^1$-norm computes the total mass present in the entire network. We also define the space of flux values at the junction O by
	\begin{align*}
		{\bm \ell}(V):=L^1(v_{\min},v_{\max};\ell^1),\quad \Vert g \Vert_{{\bm \ell}(V)} :=\sum_{j\in \mathcal{J}}\Vert g_j\Vert_{V}.
	\end{align*}
	
	On $X^{\mathsf{e}}$, we define the \emph{transport operator}
	\begin{align}\label{Trp-op}
		\tilde{A}_{\mathsf{e}}f:= \left(-v\partial_x f_j- q_j(\cdot,v)f_j\right)_{j\in \mathcal{J}},\quad f\in {\bm D}(\tilde{A}_{\mathsf{e}}):= \big\{ f \in X^{\mathsf{e}}: \partial_x f\in X^{\mathsf{e}}\big\}.
	\end{align}
	
	To incorporate memory effects arising from the delays in \eqref{Eqt2}, we define the history space
	\begin{align*}
		X^{o}:= \big\{ (\varphi_j)_{j\in \mathcal{J}} : \varphi_j \in L^1(-r_j,0;V),\ \sum_{j\in \mathcal{J}} \Vert \varphi_j\Vert_{L^1(-r_j,0;V)} < \infty \big\},
	\end{align*}
	endowed with the norm $\Vert \varphi\Vert_{X^o}:=\sum_{j\in \mathcal{J}} \Vert \varphi_j\Vert_{L^1(-r_j,0;V)}$. The space $X^{o}$ stores the past states of the boundary fluxes over the delay intervals $[-r_j,0]$.
	
	On $X^{o}$, we introduce the \emph{delay operator} ${L}:= \operatorname{diag}({L}_j)$ defined component-wise for $j\in \mathcal{J}$ by
	\begin{equation}\label{delay-operator}
		{L}_{j}(\varphi_j)(v)=\frac{1}{v}\int_{-r_j}^{0}{\rm d}\eta_j(\theta)\int_{v_{\min}}^{v_{\max}} \beta_j(v,v')v'\varphi_j(\theta,v'){\rm d}v',
	\end{equation}
	for  $\varphi_j\in L^1(-r_j,0;V)$ and $v\in [v_{\min}, v_{\max}]$,
	representing the memory-driven scattering process at the junction.
	
	%Now, let $z(t,x,v):=(z_j(t,x,v))_{j\in \mathcal{J}}$ denote the density distribution, and
    Now, we define the boundary values of the state $z_j(t,x,v)$ at the endpoints $l_j$ as
	\begin{align*}
		\check{z}_j(t,\cdot):=z(t,l_j,\cdot), \qquad \forall j\in \mathcal{J},\ t\ge 0.
	\end{align*}
	For each $t\ge 0$, we consider the history segment
	\begin{align*}
		\check{z}^t_j:[-r_{j},0] \ni \theta\to V,\quad \check{z}^t_j(\theta,\cdot):=\tilde{z}_j(t+\theta,\cdot), \qquad \forall j\in \mathcal{J}.
	\end{align*}
	Following standard practice for delay equations, we denote by $\check{z}(t)$ and $\check{z}^t$ the present and history of the density distribution $z$ at $l_j$, respectively. More precisely, we have
	\begin{align*}
		z(t):=(z_j(t,\cdot,\cdot))_{j\in \mathcal{J}}\in X^{\mathsf{e}},\quad
		\check{z}(t):=(\check{z}_j(t,\cdot))_{j\in \mathcal{J}}\in {\bm \ell}(V), \quad
		\check{z}^t:=(\check{z}^t_j(\cdot,\cdot))_{j\in \mathcal{J}}\in X^{o}.
	\end{align*}
	The history segment $\check{z}^t$ satisfies the following boundary evolution equation
	\begin{align}\label{shift}
		\begin{cases}
			\dot{y}(t,\cdot)= \tilde{A}_{\theta} y(t,\cdot),& t> 0,\quad y(0,\cdot)=\varphi,\\
			y(t,0)= \tilde{z}(t), &t> 0,
		\end{cases}
	\end{align}
	where $\tilde{A}_\theta:{\bm D}(\tilde{A}_\theta)\to X^{o}$ is the operator defined by
	\begin{align*}%\label{Q_m}
		\tilde{A}_\theta \varphi:=  \partial_\theta\varphi, \quad
		\varphi\in  {\bm D}(\tilde{A}_{\theta}):= \big\{ \varphi \in X^{o}: \partial_\theta \varphi\in X^{o}\big\}.
	\end{align*}
	
	The typical approach to the study of delay differential equations is to embed the underlying dynamical process into certain product space. For this purpose, we define
	\begin{align*}
		X :=X^{\mathsf{e}}\times X^{o},\qquad \left\|(f, \varphi)^{\top}\right\|:=\|f\|_{X^{\mathsf{e}}}+\|\varphi\|_{X^{o}},
	\end{align*}
	and the augmented state
	\begin{align*}
		\zeta(t)=(z(t), \check{z}^{t})^{\top},\qquad t\ge 0.
	\end{align*}
	
	Using the above notation,  system \eqref{Eqt1}-\eqref{Eqt2} can be reformulated as the boundary control problem
	\begin{align}\label{Main-system1}
		\begin{cases}
			\dot{\zeta}(t)= \tilde{A}\zeta(t),& t> 0,\\
			G \zeta(t)=P \zeta(t)+Ku(t), & t> 0,
		\end{cases}
	\end{align}
	with initial conditions $z(0)=(f_j)_{j\in \mathcal{J}}\in X^{\mathsf{e}}$, $\check{z}^0=(\varphi_j)_{j\in \mathcal{J}}\in X^{o}$, and input function $u(t):=u(t,\cdot)\in L^p(\R_+;V)$.
	Here, $\tilde{A}:{\bm D}(\tilde{A})\to X$ is the operator defined by
	\begin{align}\label{A_M}
		\tilde{A}:={\rm diag}(\tilde{A}_{\mathsf{e}},\tilde{A}_{\theta}), \qquad {\bm D}(\tilde{A}):={\bm D}(\tilde{A}_{\mathsf{e}})\times {\bm D}(\tilde{A}_{\theta}).
	\end{align}
	The operators  $G,P:{\bm D}(\tilde{A}) \to\bm{\ell}(V)\times \bm{\ell}(V)$ and $K:V\to \bm{\ell}(V)\times \bm{\ell}(V)$ are given by
	\begin{align}\label{G,Ga}
		G:=\begin{pmatrix}
			\delta_{0}^{X^{\mathsf{e}}} & 0 \\
			0 & \delta_{0}^{X^{o}}
		\end{pmatrix},\quad
		P:= \begin{pmatrix}
			0 & 	{L} \\
			\calF & 0
		\end{pmatrix},\quad K:= \begin{pmatrix}
			M \\
			0
		\end{pmatrix},
	\end{align}
	where for any $f\in {\bm D}(\tilde{A}_{\mathsf{e}})$ and any $\varphi\in {\bm D}(\tilde{A}_{\theta})$,
	\begin{align*}%\label{Delta}
		\mathcal{F}f:=M (f_j(l_j,\cdot))_{j\in \mathcal{J}},\quad
		\delta_{0}^{X^{\mathsf{e}}}f:=f(0,\cdot),\quad
		\delta_{0}^{X^{o}}\varphi=\varphi(0,\cdot).
	\end{align*}
      Here $M:=(w_{ij})_{i,j\in \mathcal{J}}$ denotes the network routing matrix and $\delta_{0}^E$ (for $E=X^{\mathsf{e}}$ or $X^{o}$) denotes the Dirac mass at $0$.

	\begin{remark}
     The abstract boundary condition $G \zeta(t) = P \zeta(t)+K u(t)$ provides an operatorial mathematical representation of the physical TCs \eqref{Eqt2}. This reformulation enables direct application of our theoretical framework to investigate the ISS for the kinetic transport network system \eqref{Eqt1}-\eqref{Eqt2}.
	\end{remark}
	
	In the sequel, we establish a complete characterization for the exponential $L^p$-ISS of system \eqref{Eqt1}-\eqref{Eqt2}. To this ends, we introduce the following parameters
	\begin{equation*}
		\underline{l}:= \inf_{j\in \mathcal{J}}l_j,\  \bar{l}:=\sup_{j\in \mathcal{J}}l_j,\
		\bar{r}:=\sup_{j\in \mathcal{J}}r_j, \
		\bar{\beta}:= \sup_{j\in \mathcal{J}}\Vert\beta_j\Vert_{L^{\infty}} ,\  \overline{{\rm Var}}:=\sup_{j\in \mathcal{J}}{\rm Var}(\eta_j;-\bar{r},0).
	\end{equation*}
	Furthermore, we consider the operator
	\begin{align}\label{Matrix-H}
		\Lambda:{\bm \ell}(V)\to {\bm \ell}(V), \quad
		\Lambda g :={L} e_0 \calF\Xi_0 g,
	\end{align}
	for $\lambda \in \C$ and a.e. $\theta \in [-\bar{r},0]$, $v \in [v_{\min},v_{\max}]$, $x \in [0,\bar{l}]$, where for  $ g,h\in {\bm \ell}(V)$, the auxiliary operators $\Xi_\la:{\bm \ell}(V)\to X^{\mathsf{e}}$ and $e_\la:{\bm \ell}(V)\to X^{o}$ are defined by
	\begin{subequations}\label{Dirichlet-operators}
		\begin{align}
			(\Xi_\la g)(x,v)&:=\operatorname{diag}\left({{\exp}}{\left(-\int_0^x \frac{\lambda + q_j(y,v)}{v} {\rm d}y  \right)}\right)_{j\in \mathcal{J}}g(v),\\
			(e_\la h )(\theta,v)&:= e^{\la\theta} h(v).
		\end{align}
	\end{subequations}

	The necessary and sufficient conditions for the exponential $L^p$-ISS of the kinetic transport network system \eqref{Eqt1}-\eqref{Eqt2} is stated in the following proposition.
	
	\begin{theorem}\label{Result1}
		Let Assumptions \ref{Mainassump} and \ref{Mainassump1} be satisfied, and let $p\in [1,\infty]$. Assume that
		\begin{enumerate}
			\item[(a)]  $\bar{l},\ \bar{r},\ \bar{\beta},\ \overline{{\rm Var}}<\infty$ and  $\underline{l} >0$.
			\item[(b)] The Stieltjes measure ${\rm d}\eta_j(\theta)$ is a positive operator-valued measure for all $j\in \mathcal{J}$.
		\end{enumerate}
		Then, system \eqref{Eqt1}-\eqref{Eqt2} is exponentially $L^p$-ISS if and only if
		\begin{align}\label{Cond-ISS}
			r(\Lambda) < 1.
		\end{align}
	\end{theorem}
\begin{remark}
		Theorem  \ref{Result1} provides a full spectral characterization of exponential $L^p$-ISS for the kinetic transport network system \eqref{Eqt1}-\eqref{Eqt2}.   The condition $r(\Lambda)<1$ quantifies the net interaction strength at the central junction and ensures that no amplification loop is formed through the combined effects of scattering, network routing, and distributed delays. It is worth noting that the criterion requires no rationality assumptions on length or delay ratios. For $p = \infty$, it reduces to a SGC, guaranteeing robust stability with respect to bounded disturbances entering at the junction.
	\end{remark}
	 \emph{Proof of Theorem  \ref{Result1}:\enspace}
		We apply Theorem \ref{Main2} for the proof. In fact, it follows from Lemma \ref{resolvent-estimates} that restriction $A := \tilde{A}\vert_{\ker G}$ is a resolvent positive operator and satisfies the resolvent inverse estimate \eqref{inverse-estimate-transport}. Moreover, the operator $G$ defined in \eqref{G,Ga} is clearly surjective. A direct computation shows that the Dirichlet operator $D_\lambda$ associated with $(\tilde{A}, G)$ is given by
		\begin{align*}%\label{Dirich}
			D_\lambda(g,h)^{\top} = \begin{pmatrix}
				\Xi_\lambda g \\
				e_\lambda h
			\end{pmatrix},\quad \forall \lambda \in \C,\ (g,h)^{\top} \in \bm{\ell}(V)\times \bm{\ell}(V),
		\end{align*}
		where $\Xi_\lambda$ and $e_\lambda$ are the operators defined in \eqref{Dirichlet-operators}. The operator $D_\lambda$ is clearly positive for all $\lambda > 0$. Furthermore, condition (b) and the positivity of $\beta_j$ imply that the delay operators $L_j$ defined in \eqref{delay-operator} are positive for each $j$, hence the operator $P$ defined in \eqref{G,Ga} is positive.
		
		Now, let $\lambda > 0$, $g\in \bm{\ell}_+(V)$, and  $h\in {\bm{\ell}_+}(V)$. Using Assumptions \ref{Mainassump} and \ref{Mainassump1}, we obtain
	  \begin{align*}
				\Vert PD_\la(g, h)^{\top}\Vert
				&\le\sum_{j\in \mathcal{J}}\int_{v_{\min}}^{v_{\max}}{L}_{j}((e_\lambda h)_j)(v)\text{d} v
				+ \sum_{i\in \mathcal{J}}\sum_{j\in \mathcal{J}}w_{ij}\displaystyle\int_{v_{\min}}^{v_{\max}} (\Xi_\lambda g)_j(l_j,v) \text{d}v\\
				&\le  \bar{\beta}v_{\max}\ln\left(\frac{v_{\max}}{v_{\min}}\right)\sup_{j\in \mathcal{J}} \left\Vert \int_{-r_j}^{0} e^{\la \theta}{\rm d}\eta_j(\theta)\right\Vert_{\calL(V)} \Vert h\Vert_{{\bm \ell}(V)}
				+  e^{\frac{\bar{\gamma} \bar{l}}{v_{\min}}}e^{\frac{-\la \underline{l} }{v_{\max}}} \Vert M \Vert  \Vert g\Vert_{{\bm \ell}(V)},
				%	=& \bar{\beta}v_{\max}\ln(\frac{v_{\max}}{v_{\min}}) \sup_{i\in \mathcal{J}}\left\Vert\int_{-r_i}^{0} e^{\la \theta}{\rm d}\eta_i(\theta)\right\Vert_{\calL(V)} \Vert h\Vert_{{\bm \ell}^{\mathsf{e}}(V)}\\
				%	&+  e^{\frac{\vert \gamma\vert l_{\max} }{v_{\min}}}e^{\frac{-\la l_{\min} }{v_{\max}}} \Vert g\Vert_{{\bm \ell}^{\mathsf{n}}(V)},
		\end{align*}
		where $\bar{\gamma}:=\max\{\vert\gamma_1\vert,\vert\gamma_2\vert\}$. Using Lemma \ref{tool}, we then obtain
		\begin{align*}
			\lim_{\la \to +\infty} \Vert P D_\la\Vert =0.
		\end{align*}
		Therefore, there exists $\lambda_1 > 0$ large enough such that $\Vert P D_{\lambda_1} \Vert <1$.
		
		All assumptions of Theorem \ref{Main1} are satisfied; consequently, the kinetic transport network system \eqref{Eqt1}-\eqref{Eqt2} admits a unique solution $z \in {C}(\R_+, X^{\mathsf{e}})$ given by
		\begin{align}\label{Explicit}
			z(t,f,\varphi,u)=\left[\zeta(t,f,\varphi,u)\right]_1,\quad \forall  t \ge 0,f \in X^{\mathsf{e}},\varphi \in X^{o},u \in L^p_{loc}(\R_+; V),
		\end{align}
		where $\left[\zeta(t, f, \varphi, u)\right]_1$ denotes the first component of the solution to \eqref{Main-system1}.
		
		Now, applying Theorem \ref{Main2}, it follows that system \eqref{Eqt1}-\eqref{Eqt2} is exponentially $L^p$-ISS if and only if $r(PD_0) < 1$, since $s(A) <0$. Using the factorization
		\begin{align}\label{factorization}
			\mathcal{I}-PD_0&=\begin{pmatrix}
				\mathcal{I} & -{L} (e_0)\\
				-\calF \Xi_0 &  \mathcal{I}
			\end{pmatrix}
			=\begin{pmatrix}
				\mathcal{I}-\Lambda & - {L} (e_0)\\
				0& 	\mathcal{I}
			\end{pmatrix}			\begin{pmatrix}
				\mathcal{I} & 0\\
				-\calF\Xi_0  &  \mathcal{I}
			\end{pmatrix},	
		\end{align}
		we conclude that $1 \in \rho(PD_0)$ if and only if $1 \in \rho(\Lambda)$. Since both $PD_0$ and $\Lambda$ are positive operators, it follows that the kinetic transport network system \eqref{Eqt1}-\eqref{Eqt2} is exponentially $L^p$-ISS if and only if \eqref{Cond-ISS} holds. This completes the proof. \hfill $\blacksquare$\par

	The following result  provides an ISS estimate for the solutions of the system \eqref{Eqt1}-\eqref{Eqt2}.
	\begin{corollary}\label{Result2}
		Let the assumptions of Theorem \ref{Result1} be satisfied. Suppose that the scattering at junctions is mass-preserving, i.e.,
		\begin{align}\label{beta}
			\int_{v_{\min}}^{v_{\max}}\beta_j(v,v'){\rm{d}} v=1,\quad \forall j\in \mathcal{J}, \,v'\in [v_{\min},v_{\max}].
		\end{align}
		If
		\begin{align}\label{C1}
			\max\left\{\frac{v_{\max}}{v_{\min}}\ \overline{{\rm Var}} ,e^{\frac{ \bar{l}\bar{\gamma}}{v_{\min}}}\Vert M \Vert\right\}<1,
		\end{align}
		then there exist positive constants $N,a $ such that
		\begin{align}\label{ISS-estimate-trp}
			\left\Vert z(t)\right\Vert_{X^{\mathsf{e}}} \leq  Ne^{-a t}\left(\Vert f\Vert_{X^{\mathsf{e}}}+\Vert\varphi\Vert_{X^o}\right)+\frac{N e^{\frac{ \bar{l}\bar{\gamma}}{v_{\min}}}\Vert M \Vert}{ac(1-C)}\Vert u\Vert_{L^\infty(\R_+;V)},
		\end{align}
		for all $t\ge 0$, $f\in X^{\mathsf{e}}$, $\varphi\in X^{o}$, and $u\in L^\infty(\R_+;V)$, where $c$ is the constant given by \eqref{constant} and
		$$C:=\max\left\{\frac{v_{\max}}{v_{\min}}\ \overline{{\rm Var}},e^{\frac{ \bar{l}\bar{\gamma}}{v_{\min}}}\Vert M \Vert\right\}.$$
	\end{corollary}
	
	\begin{proof}
		Using Assumptions \ref{Mainassump} and \ref{Mainassump1}, we obtain
		{\small
        \begin{align*}
			\Vert PD_0(g, h)^{\top}\Vert
			\le & \overline{{\rm Var}} \frac{v_{\max}}{v_{\min}}\sum_{j\in \mathcal{J}}\int_{v_{\min}}^{v_{\max}}\int_{v_{\min}}^{v_{\max}}\beta_j(v,v'){\rm d}vh_j(v')  {\rm d}v'
			+e^{\frac{-\gamma_1 \bar{l}}{v_{\min}}}\sum_{i\in \mathcal{J}}\sum_{j\in \mathcal{J}} w_{ij}\Vert g_j\Vert_V\\
			=& \overline{{\rm Var}}\ \frac{v_{\max}}{v_{\min}}\Vert h\Vert_{{\bm \ell}(V)}
			+ e^{\frac{ \bar{l}\bar{\gamma}}{v_{\min}}}\Vert M \Vert\Vert g\Vert_{{{\bm \ell}}(V)},
			%\\\le& \max\left\{\overline{{\rm Var}} \frac{v_{\max}}{v_{\min}},\frac{-\gamma_1 \bar{l}}{v_{\min}}\Vert M \Vert\right\}\Vert (g, h)^{\top}\Vert.
		\end{align*}}for all $g\in {\bm \ell}_+(V)$ and $h\in {\bm \ell}_+(V)$,
		where we used the additivity of the $L^1$-norm on the positive cone and the condition \eqref{beta} to obtain the equality in the second line. Therefore,
		\begin{align}\label{ESTIMATE}
			\Vert PD_0 \Vert\le \max\left\{\overline{{\rm Var}}\ \frac{v_{\max}}{v_{\min}},e^{\frac{ \bar{l}\bar{\gamma}}{v_{\min}}}\Vert M \Vert\right\}.
		\end{align}
		If \eqref{C1} holds, then $\Vert PD_0\Vert<1$, and hence $r(\Lambda)<1$. By Theorem \ref{Result1}, the kinetic transport network system \eqref{Eqt1}-\eqref{Eqt2} is exponentially $L^p$-ISS. Moreover, we have
		\begin{align}\label{Diri-ineq}
			\Vert D_0\Vert\le e^{\frac{ \bar{l}\bar{\gamma}}{v_{\min}}},\qquad \Vert K\Vert\le \Vert M \Vert.
		\end{align}
		Combining \eqref{ESTIMATE} and \eqref{Diri-ineq} through application of Theorem \ref{Main2}, there exist constants $N,a,c,C>0$ such that
		\begin{align*}%\label{Exp-ISS}
			\Vert \zeta(t,f,\varphi,u)\Vert_{X}\le Ne^{-a t}\left\Vert (f,\varphi)^\top\right\Vert_{X}+\frac{N e^{\frac{ \bar{l}\bar{\gamma}}{v_{\min}}}\Vert M \Vert}{ac(1-C)}\Vert u\Vert_{L^\infty(\R_+;V)},
		\end{align*}
		for all $t\ge 0$, $f\in X^{\mathsf{e}}$, $\varphi\in X^{o}$, and $u\in L^\infty(\R_+;V)$. Applying \eqref{Explicit} then yields the ISS estimate \eqref{ISS-estimate-trp}, which completes the proof.
	\end{proof}

	We conclude with two examples that illustrate our results for particular operator-valued functions $\eta_j$.
	\begin{example}
		Consider the case where
		\begin{align*}
			\emph{d}\eta_j(\theta):=\delta_{-r_j}^{X^o}(\theta)\emph{d}\theta,\qquad \forall j\in \mathcal{J},
		\end{align*}
        where {$\delta_{-r_j}^{X^o}$ is the Dirac mass at $-r_j$}. These are positive operator-valued measures with ${{\rm Var}}(\eta_j;-r_{j},0)=1$ for all $j\in \mathcal{J}$.  If $\bar{r}<\infty$, then $\overline{{\rm Var}}=1$. By Theorem \ref{Result1}, the kinetic transport network system \eqref{Eqt1}-\eqref{Eqt2} is exponentially $L^p$-ISS if and only if $r(\Lambda)<1$, where in this case
		\begin{align*}
			(\Lambda g)_i(v)=\frac{1}{v}\int_{v_{\min}}^{v_{\max}}\beta_j(v,v')v'\sum_{j\in \mathcal{J}}w_{ij}
			{{\exp}}\left(-\int_0^{l_j} \frac{ q_j(y,v')}{v'} {\rm d}y  \right)g(v'){\rm d}v',
		\end{align*}
		for all $i\in \mathcal{J}$, $g\in {\bm \ell}(V)$, and a.e. $v\in [v_{\min},v_{\max}]$. The spectral radius $r(\Lambda)<1$ can be estimated by
		\begin{align*}
			r(\Lambda)\le \Vert \Lambda\Vert\le \bar{\beta}v_{\max}\ln\left(\frac{v_{max}}{v_{\min}}\right)e^{\frac{\bar{\gamma}\bar{l}}{v_{\min}}}\Vert M \Vert.
		\end{align*}
		Any choice of parameters satisfying $	\bar{\beta}v_{\max}\ln(\frac{v_{max}}{v_{\min}})e^{\frac{\bar{\gamma}\bar{l}}{v_{\min}}}\Vert M \Vert<1$
		ensures $r(\Lambda)<1$.
	\end{example}
	
	\begin{example}
		Consider the case where
		\begin{align*}
			 \emph{d}\eta_j:=e^{\vartheta_j\theta}\emph{d}\theta,\qquad \forall j\in \mathcal{J},
		\end{align*}
        		where $0\neq\vartheta_j\in \R$. Observe that ${\rm d}\eta_j$ are positive operator-valued measures with ${{\rm Var}}(\eta_j;-r_{j},0)=\frac{1}{\vartheta_j}(1-e^{-\vartheta_jr_j})$ for all $j\in \mathcal{J}$. If $\bar{r}<\infty$ and $0< \underline{\vartheta}:=\inf_{j\in \mathcal{J}}\vartheta_j $, then $\overline{{\rm Var}}\le \bar{r}$. By Theorem~\ref{Result1}, the kinetic transport network system \eqref{Eqt1}-\eqref{Eqt2} is exponentially $L^p$-ISS if and only if $r(\Lambda)<1$, where in this case
		\begin{align*}
			(\Lambda g)_i(v)=\frac{1}{\vartheta_j}(1-e^{-\vartheta_jr_j})\frac{1}{v}\int_{v_{\min}}^{v_{\max}}\beta_j(v,v')v'\sum_{j\in \mathcal{J}}w_{ij} {{\exp}}\left(-\int_0^{l_j} \frac{ q_j(y,v')}{v'} {\rm d}y  \right)g(v'){\rm d}v',
		\end{align*}
		for all $i\in \mathcal{J}$, $g\in {\bm \ell}(V)$, and a.e. $v\in [v_{\min},v_{\max}]$. If the condition \eqref{beta} holds, then the spectral radius $r(\Lambda)<1$ can be estimated by
		 \begin{align*}
			 r(\Lambda)\le \Vert \Lambda\Vert\le \bar{r}\frac{v_{max}}{v_{\min}}e^{\frac{\bar{\gamma}\bar{l}}{v_{\min}}}\Vert M \Vert.
		 \end{align*}
		Thus, by choosing parameters such that $\bar{r}\frac{v_{max}}{v_{\min}}e^{\frac{\bar{\gamma}\bar{l}}{v_{\min}}}\Vert M \Vert<1$, we obtain $r(\Lambda)<1$. Furthermore, if $\Vert M \Vert<e^{-\frac{\bar{\gamma}\bar{l}}{v_{\min}}}$ and $ \bar{r}<\frac{v_{\min}}{v_{\max}}$, then it follows from Corollary \ref{Result2} that the ISS estimate \eqref{ISS-estimate-trp} holds with $C=\max\{\frac{v_{\max}}{v_{\min}}\ \bar{r},e^{\frac{ \bar{l}\bar{\gamma}}{v_{\min}}}\Vert M \Vert\}$.
	\end{example}

	\appendix
	
	\section{Technical lemmas}\label{App1}
	Here we provide the proofs for the technical results used in the proof of Theorem \ref{Result1}.
	
	\begin{lemma}\label{resolvent-estimates}
		Let Assumption \ref{Mainassump} be satisfied, and suppose $\bar{l},\,\bar{r}<\infty$. Then, ${A}:=\tilde{A}|_{ \ker G}$  {is a} resolvent positive operator such that
		\begin{align}\label{inverse-estimate-transport}
			\|R(\lambda,{A})(f,\varphi)^{\top}\|\geq c\|(f,\varphi)^{\top}\|,
		\end{align}
		for all $(f,\varphi)^{\top}\in X_+$ and $\lambda>0\vee ({-\gamma_2})$, where the positive constant $c$ is given by
		\begin{align}\label{constant}
			c:=\inf\left\{\frac{1}{\la+\gamma_2}\left[ {{\exp}}\left(\frac{(\gamma_2+\la)}{v_{\min}}\bar{\alpha}\right)-1\right], \frac{1}{\lambda}	\left(1-{{\exp}}(-\lambda (\bar{k}+\bar{r}))\right)\right\}.
		\end{align}
	\end{lemma}
	
	\begin{proof} %:=(A_{\mathsf{e}})|_{ \ker \calF }
		Let $\tilde{A}_{\mathsf{e}}$ be the operator defined in \eqref{Trp-op}. Then, the domain of the restriction operator ${A}_{\mathsf{e}}:=(\tilde{A}_{\mathsf{e}})\vert_{\ker \delta_{0}^{X^{\mathsf{e}}}}$ is given by
		%\begin{align*}
		${\bm D}(A_{\mathsf{e}})= \{ f \in {\bm D}(\tilde{A}_{\mathsf{e}}): f(0,\cdot)=0 \}.$
		%\end{align*}
		A direct computation shows that its resolvent is given by
		\begin{align*}
			\big(R(\lambda,A_{\mathsf{e}}) f\big)_j(x,v)
			= \frac{1}{v} \int_0^x {{\exp}}\left( -\int_y^x \frac{\lambda + q_j(s,v)}{v}  {\rm{d}}s \right) f_j(y,v)  {\rm d}y,
		\end{align*}
		for all $j\in \mathcal{J}$, $\la\in \C$, $f\in X^{\mathsf{e}}$, and a.e. $(x,v)\in \Omega_j$.
		
		Now, let $f\in X_+^{\mathsf{e}}$ and {$\la>-\gamma_2$}, where $\gamma_2$ is the constant from \eqref{gamma}. Using the additivity of the $\Vert\cdot\Vert_{X^{\mathsf{e}}}$-norm on $X_+^{\mathsf{e}}$, and Fubini's theorem, we obtain
		\begin{align*}
			\Vert R(\la,{A}_{{\mathsf{e}}})f\Vert_{X}
			%= &\sum_{j\in \mathcal{J}}\int_{v_{\min}}^{v_{\max}}\int_{0}^{l_j} \quad (R(\la,{A}_{{\mathsf{e}}})f)_j(x,v) {\rm d}x{\rm d}v\\
			\ge&  \sum_{j\in \mathcal{J}}\int_{v_{\min}}^{v_{\max}}\int_{0}^{l_j} \int_{0}^{x}		{{\exp}}\left(\frac{-(\gamma_2 +\la)}{v}(x-y)\right)\frac{1}{v}f_j(y,v){\rm d}y{\rm d}x{\rm d}v\\
			\ge &	\frac{ 1}{\la +\gamma_2}\sum_{j\in \mathcal{J}}\int_{v_{\min}}^{v_{\max}}\int_{0}^{l_j}     \left[{{\exp}}\left(\frac{(\gamma_2+\la)}{v_{\min}}y\right)-1 \right]f_j(y,v){\rm d}y{\rm d}v.
		\end{align*}
		By the First Mean Value Theorem for Integrals (see, e.g., \cite[Thm. 85.6]{Heuser} and \cite[Rem. 2.1]{Diethelm}), there exists $\alpha_j\in (0,l_j)$ such that
		\begin{align*}
			\displaystyle\int_{0}^{l_j}\left[{{\exp}}\left(\frac{(\gamma_2+\la)}{v_{\min}}y\right)-1\right] f_j(y,\cdot){\rm d}y=
			\left[ {{\exp}}\left(\frac{(\gamma_2+\la)}{v_{\min}}\alpha_j\right)-1\right] \int_{0}^{l_j}f_j(y,\cdot){\rm d}y,
		\end{align*}
		for all $j\in \mathcal{J}$, $ f_j(\cdot,v)\in L^1_+(0,l_j)$, and a.e. $v\in [v_{\max},v_{\min}]$. Therefore,
		\begin{align}\label{inverse-A}
			\Vert R(\la,{A}_{{\mathsf{e}}})f\Vert_{X^{\mathsf{e}}}\ge \frac{1}{\la+\gamma_2}\left[ {{\exp}}\left(\frac{(\gamma_2+\la)}{v_{\min}}\bar{\alpha}\right)-1\right]\Vert f\Vert_{X^{\mathsf{e}}},
		\end{align}
		for all {$\la>-\gamma_2$} and $f\in X_+^{\mathsf{e}}$, where $\bar{\alpha}:=\sup_{j\in \mathcal{J}} \alpha_j$.
		
		Next, consider ${A}_{\theta}:=(A_{\theta})\vert_{\ker \delta_{0}^{X^{o}} }$. Its resolvent is given by
		\begin{align*}
			(R(\la,{A}_{\theta})\varphi)_j(\theta,v)=\int_{\theta}^{0} e^{(\theta-\sigma)\la}\varphi_j(\sigma,v){\rm d}\sigma,
		\end{align*}
		for all $j\in \mathcal{J}$, $\la\in \C$, $\varphi\in X^{o}$, and a.e. $(\theta,v)\in \Theta_j$. A similar argument, using the additivity of the $\Vert\cdot\Vert_{X^{o}}$-norm on $X_+^{o}$, Fubini's theorem, and the First Mean Value Theorem for Integrals, yields
		\begin{align}\label{inverse-Q}
			\Vert R(\lambda,{A}_{\theta})\varphi\Vert_{X^{o}}\ge  \frac{1}{\lambda}	\left(1-{{\exp}}(-\lambda (\bar{k}+\bar{r}))\right)\Vert \varphi\Vert_{X^{o}},
		\end{align}
		for all $\lambda>0$ and $\varphi\in X_+^{o}$, where $\bar{k}:=\sup_{j\in \mathcal{J}} k_j$, for some $k_j\in (0,r_j)$.
		
		Now, observe that the operator ${A}$ has the diagonal structure
		\begin{align}\label{A-diag}
			{A}=\operatorname{diag}({A}_{\mathsf{e}},{A}_{\theta}).
		\end{align}
		From this, it follows that $s({A})=-\infty$ and $R(\la,{A})\ge 0$ for all $\la>-\infty$; that is, $A$ is resolvent positive. Combining the estimates \eqref{inverse-A} and \eqref{inverse-Q}, we obtain the resolvent inverse estimate \eqref{inverse-estimate-transport}. This completes the proof.
	\end{proof}
	
	\begin{lemma}\label{tool}
		Let Assumption \ref{Mainassump1} be satisfied and assume that for each $ j \in \mathcal{J} $, the Stieltjes measure $  \emph{d} \eta_j(\theta)$ is a positive operator-valued measure. If $ \bar{r}, \overline{{\rm Var}}< \infty $, then
		\begin{align*}
			\lim_{\lambda \to \infty} \sup_{i \in \mathcal{M}} \left\| \int_{-r_i}^{0} e^{\lambda \theta} \emph{d}\eta_i(\theta) \right\|_{\mathcal{L}(V)} = 0.
		\end{align*}
	\end{lemma}
	\begin{proof}
		The key idea is to use the positivity of the measure to bound the operator norm. Since ${\rm d}\eta_j(\theta)$ is a positive operator-valued measure, it follows that
		\begin{align*}
			0 \le \int_{-r_j}^{0} e^{\lambda \theta} {\rm d}\eta_j(\theta) \le \int_{-\bar{r}}^{0} e^{\lambda \theta} {\rm d}\eta_j(\theta) \le \int_{-\bar{r}}^{0} {\rm d}\eta_j(\theta),\quad \forall j\in \mathcal{J},\lambda>0.
		\end{align*}
		%for all $j\in \mathcal{J}$ and $\lambda>0$.
		Therefore, taking the operator norm and then the supremum over $ j$, we obtain
		\begin{align}\label{uniform-bound}
			\sup_{j \in \mathcal{J}} \left\| \int_{-r_j}^{0} e^{\lambda \theta}  {\rm d} \eta_j(\theta) \right\|_{\mathcal{L}(V)} \le \sup_{j \in \mathcal{J}} \left\|\displaystyle \int_{-\bar{r}}^{0} {\rm d}\eta_j(\theta) \right\|_{\mathcal{L}(V)} \le \overline{{\rm Var}} < \infty.
		\end{align}
		
		Now, using Assumption \ref{Mainassump1}, it follows from \cite[Estimate~(3.12)]{BSc} that for any $ 0 < \epsilon < \bar{r}$ and for all sufficiently large $\lambda $, the following holds for all $j \in \mathcal{J} $:
		\begin{align}\label{key-estimate}
			\left\| \int_{-\bar{r}}^{0} e^{\lambda \theta} {\rm d}\eta_j(\theta) \right\|_{\mathcal{L}(V)} \le & {\rm Var} (\eta_j, -\epsilon, 0)
			+ {\rm Var} (\eta_j, -\bar{r}, -\epsilon) e^{-\lambda \epsilon}.
		\end{align}

		Since ${{\rm Var}}(\eta_j, -\bar{r},0) \le \overline{{\rm Var}}$ for all $j \in \mathcal{J}$, the estimate \eqref{uniform-bound} and \eqref{key-estimate} yield
		\begin{align*}
			\sup_{j \in \mathcal{J}} \left\| \int_{-r_j}^{0} e^{\lambda \theta} {\rm d}\eta_j(\theta) \right\|_{\mathcal{L}(V)}  \le \sup_{j \in \mathcal{J}} \left\| \int_{-\bar{r}}^{0} e^{\lambda \theta} {\rm d}\eta_j(\theta) \right\|_{\mathcal{L}(V)}
			\le \sup_{j \in \mathcal{J}} {\rm Var} (\eta_j, -\epsilon, 0) + \overline{{\rm Var}} \, e^{-\lambda \epsilon}.
		\end{align*}
		
		Let $\epsilon_0 > 0$ be arbitrary. By the properties of the total variation, we can choose $ \epsilon > 0 $ sufficiently small such that $$ \sup_{j \in \mathcal{J}} {\rm Var} (\eta_j, -\epsilon, 0) < \frac{\epsilon_0}{2}. $$  For this fixed $ \epsilon$, we then choose $ \lambda $ large enough so that $ \overline{{\rm Var}} \ e^{-\lambda \epsilon} < \frac{\epsilon_0}{2}$. Therefore, for all $ \lambda $ sufficiently large,
		 \begin{align*}
		 \sup_{j \in \mathcal{J}} \left\| \int_{-r_j}^{0} e^{\lambda \theta} {\rm d}\eta_j(\theta) \right\|_{\mathcal{L}(V)} < \epsilon_0,
		 \end{align*}
		which completes the proof.
	\end{proof}

%%%%%%%%%%%%%%%%%%%%%%%%%%%%%%%%%%%%%%%%%%%%%%
\end{document}